\documentclass[11pt, reqno]{amsart}
\usepackage{fullpage}

\newif\ifArxiv
\Arxivtrue

 \newif\ifHideFoot
\HideFoottrue

\ifHideFoot
\else
\usepackage[notref,notcite]{showkeys}
\fi

\usepackage{amssymb,amsmath,amsthm,amscd,mathrsfs,graphicx, color}
\usepackage[cmtip,all,matrix,arrow,tips,curve]{xy}
 \usepackage{hyperref}
\usepackage[usenames,dvipsnames]{xcolor}
\usepackage{xypic}
\hypersetup{colorlinks=true,citecolor=ForestGreen,linkcolor=Maroon,urlcolor=NavyBlue}
 \usepackage{mathtools}
\usepackage{mathpazo}
\usepackage[normalem]{ulem}
\usepackage{cleveref}
\usepackage{enumitem}
\usepackage{tikz-cd}
  
\numberwithin{equation}{section}
\newtheorem{teo}{Theorem}[section]
\newtheorem{pro}[teo]{Proposition}
\newtheorem{lem}[teo]{Lemma}

\newtheorem{teoalpha}{Theorem}

\theoremstyle{definition}

\newtheorem{dfn}[teo]{Definition}
\theoremstyle{remark}
\newtheorem{rem}[teo]{Remark}

\ifHideFoot

\newcommand{\Yano}[1]{}
\newcommand{\Shend}[1]{}

\else

\newcommand{\marg}[1]{\normalsize{{
\color{red}\footnote{{\color{blue}#1}}}{\marginpar[\vskip
-.25cm{\color{red}\hfill\tiny\thefootnote$\implies$}]{\vskip
-.2cm{\color{red}$\impliedby$\tiny\thefootnote}}}}}
\newcommand{\Yano}[1]{\marg{(Yano) #1}}
\newcommand{\Shend}[1]{\marg{(Shend) #1}}

\fi

\newcommand{\m}[1]{\mathcal{#1}}

\newcommand{\X}{\mathcal{X}}

\newcommand{\Z}{\mathcal{Z}}
\newcommand{\Y}{\mathcal{Y}}

\title[Positivity for Hodge modules on stacks]{Positivity in the context of Hodge modules and Higgs bundles on Deligne--Mumford stacks}

\author{Sebastian Casalaina-Martin}
\address{University of Colorado, Department of Mathematics, 
Boulder, CO 80309, USA }
\email{casa@math.colorado.edu}

\author{Shend Zhjeqi}
\address{University of Michigan, Department of Mathematics, 
Ann Arbor, MI 48109, USA }
\email{shendzh@umich.edu}

\thanks{Research of the first named author is supported in part by a grant from the Simons Foundation (SFI-MPS-TSM-00013682). The second named author was partially supported by the Simons Collaboration grant Moduli of Varieties.}

\date{\today}

\begin{document}

\begin{abstract}
We generalize positivity results due to Popa--Wu and Popa--Schnell for Hodge modules and Higgs bundles on smooth projective varieties to the case of smooth proper DM stacks admitting projective coarse moduli spaces.  This paper is the first in a series  aiming to generalize results of Popa--Schnell and Wei--Wu on Viehweg hyperbolicity to the setting of DM stacks, and in particular, to certain KSBA moduli spaces.  
\end{abstract}

\maketitle

\section*{Introduction}

In recent work, Popa--Schnell \cite{PS17} used certain positivity results for Hodge modules and Higgs bundles on smooth projective varieties \cite{PW16,PS17}  to study the log Kodaira dimension of smooth projective parameter spaces of  varieties.   These type of positivity results go back to  Griffiths' study of semi-positive metrics on vector bundles obtained as 
the lowest term in the Hodge filtration of a variation of Hodge structures, and have been generalized in a number of ways since then by 
Fujita  \cite{fujita_fiber_spaces},
Kawamata \cite{Kawamata},
Fujino--Fujisawa \cite{FF14variationsHMS}, and 
Fujino--Fujisawa--Saito  \cite{FFS14Remarks}. Zuo's work   \cite{Zuo00negativity}  
reframes these results in terms of  kernels of Kodaira--Spencer type maps associated to meromorphic connections with log-poles and unipotent monodromy, providing an approach that  has proven quite useful in applications, and was explored further in Brunebarbe  \cite{brunebarbe18}.  
    Strategies  for using these positivity results to study the geometry of base spaces of families of varieties were developed in the work of Viehweg--Zuo 
\cite{VZ01isotriv, VZ02, VZ03}, 
who, in particular,  use these techniques to establish the existence of what are now called Viehweg--Zuo sheaves.

To apply these techniques to study the log Kodaira dimension of moduli spaces, ideally one would like to have such techniques available for smooth algebraic stacks, and, at the very least, for  smooth proper DM stacks that admit projective coarse moduli spaces.  In this paper, we consider Hodge modules and Higgs bundles on such stacks, and generalize the positivity results of Popa--Wu and Popa--Schnell in this setting.   The present article is the first in a series with the aim of generalizing \cite{PS17} and \cite{WW23} to the case of Deligne--Mumford stacks, with the ultimate motivation being to show that certain KSBA moduli stacks, all of which are DM stacks, but not all of which are schemes, are of log general type.

To begin, one should specify what is meant by positivity in this setting.   We do this in \Cref{S:positivityI}.  In short, the  stacks we consider all admit finite flat covers from smooth projective varieties (see \S \ref{S:DM-intro}),  and a sheaf on such a  stack has the positivity property if the pull back  via this finite flat cover has the positivity property.   This turns out to be a convenient condition to work with, and implies that if a sheaf on the stack descends to a sheaf on the coarse moduli space with a  given positivity property, then the sheaf on the stack has that positivity property.  
 The next preliminary is to specify what is meant by   $D$-modules and Hodge modules on such stacks, which we do in \Cref{S:DM-HM-stacks}.
In general, if one wants to obtain a $6$-functor formalism, there are more sophisticated approaches to Hodge modules on stacks, e.g., \cite{tubach_2024}, or \cite{paulin13} for $D$-modules; however, for our purposes here, where we only need to pull back  pure polarizable Hodge modules along non-characteristic morphisms,  the more elementary approach we take here, closely following say the treatment of $D$-modules on stacks in \cite{BDstacks}, suffices.

Now suppose $\mathcal X$ is a smooth proper DM stack over $\mathbb C$ with projective coarse moduli space and 
  $\mathsf M$ is a pure polarizable Hodge module on $\mathcal X$ with associated filtered  regular holonomic left $D_{\mathcal X}$-module $(\mathcal M,F_\bullet )$.  In this situation we have a natural Kodaira--Spencer type $\mathcal O_{\mathcal X}$-module homomorphism 
$$
\theta_p:\operatorname{gr}_p^F\mathcal M\longrightarrow  \operatorname{gr}_{p+1}^F \mathcal M\otimes \Omega^1_{\mathcal X}.
$$
We denote the kernel by $\mathcal K_p(\mathsf M):=\ker\theta_p$, which is a coherent $\mathcal O_{\mathcal X}$-module. 
Our first result is a generalization of \cite[Thm.~A]{PW16}:

\begin{teoalpha}\label{TA:PW-TA}
 Let $\mathcal X$ be a smooth proper integral DM stack over $\mathbb C$ with projective coarse moduli space.  If $\mathsf M$ is a pure polarizable  Hodge module on $\mathcal X$ with strict support $\mathcal X$, then for all $p$, the torsion-free sheaf ${\mathcal K}_p(\mathsf M)^\vee$ is  weakly positive (\Cref{D:PotWP}).
\end{teoalpha}

\Cref{TA:PW-TA} is proved in \Cref{T:PW-TA}.  The strategy of our proof is to show that there exists a finite flat non-characteristic cover of $\mathcal X$ with respect to $(\mathcal M,F_\bullet)$ by a smooth projective variety  (\Cref{T:nccover}), and to then reduce to \cite[Thm.~A]{PW16} on varieties (which in turn builds on results in \cite{Zuo00negativity, brunebarbe18})   
by pull back   along the finite flat non-characteristic morphism, using some results about positivity on DM stacks that we establish in \S\ref{S:positivityI}.  We believe the result on non-characteristic covers could be of independent interest; the argument is to show that the Bertini type argument in \cite[Thm.~2.1]{KV04}, establishing the existence of a finite flat cover of the stack by a smooth projective variety, can be modified to provide the non-characteristic cover. 
We also use a similar strategy to prove a related theorem, \Cref{T:PW-4.8},  on graded logarithmic Higgs bundles, which generalizes  \cite[Thm.~4.8]{PW16} (which, similarly, builds on results in \cite{Zuo00negativity, brunebarbe18}).

\medskip 
We use \Cref{TA:PW-TA} ($=$\Cref{T:PW-TA}) and \Cref{T:PW-4.8} to show that under certain hypotheses on a Hodge module and Higgs bundle on $\mathcal X$, one obtains Viehweg--Zuo sheaves; 
i.e., one obtains that  tensor powers of $\Omega^1_{\mathcal X}$, and certain  twists of those sheaves, contain subsheaves with big determinant.  The version for Higgs bundles is as follows, and generalizes \cite[Thm.~3.7]{PS17}:

\begin{teoalpha}\label{TA:PS-3.7/19.1}
Let $\mathcal X$ be a smooth proper integral DM stack over $\mathbb C$ with projective coarse moduli space, let $\mathcal D\subseteq \mathcal X$ be an snc divisor, and 
let $(\mathcal E_\bullet, \theta_\bullet)$ be a graded logarithmic Higgs bundle on $\mathcal X$ that 
extends a polarizable variation of Hodge structure of weight $\ell$ on $\mathcal X-\mathcal D$. 
If $\mathcal E_\bullet$ admits a large graded submodule for some divisor $\mathcal B\subseteq \mathcal D$ 
(\Cref{D:large-sub-Higgs}), 
then there exist a torsion-free coherent sheaf $\mathcal H$ on $\mathcal X$ with big determinant, 
and an integer  $s$ with  $1\le s\le \ell$,  together with an inclusion
\begin{equation}\label{E:TA:PS-3.7/19.1}
\xymatrix{\mathcal H \ar@{^(->}[r] & \left( \Omega^1_{\mathcal X}(\log \mathcal B)\right)^{\otimes s}.
}
\end{equation}
\end{teoalpha}

This is proved in \Cref{T:PS-3.7/19.1} (see also \Cref{L:pbig-prpty}\ref{L:pbig-prpty-det} and \Cref{R:PThmB}).  Once one has the positivity results of \Cref{T:PW-4.8}, the proof  of  \Cref{T:PS-3.7/19.1} is essentially identical to that of \cite[Thm.~3.7]{PS17}, which in turn is based off a  strategy developed in \cite{VZ02}. 
  The related theorem for Hodge modules, as opposed to Higgs bundles, derived from \Cref{TA:PW-TA} ($=$\Cref{T:PW-TA})  and generalizing \cite[Thm.~3.5]{PS17}, is our \Cref{T:PS-3.5/18.4}.

We note that in the special case where $\mathcal X=[V/G]$ for some smooth projective variety $V$ and some finite group $G$ acting on $V$, then the main results in our paper (\Cref{TA:PW-TA} $=$ \Cref{T:PW-TA}, \Cref{T:PW-4.8}, \Cref{T:PS-3.5/18.4}, and \Cref{TA:PS-3.7/19.1} $=$ \Cref{T:PS-3.7/19.1}) follow essentially immediately from those in \cite{PW16, PS17} via the \'etale cover $V\to [V/G]$.  Note also that in the  case $\mathcal X=[V/G]$, Campana--P\u aun have shown that the cokernel of \eqref{E:TA:PS-3.7/19.1} has pseudo-effective determinant (the determinant of $\mathcal H$ is assumed to be big), so that taking determinants in \eqref{E:TA:PS-3.7/19.1} one obtains that $K_{\mathcal X}+\mathcal B$ is big.   In a forthcoming paper we extend the Campana--P\u aun result to all smooth proper integral DM stacks over $\mathbb C$ with projective coarse moduli space.

\subsection*{Acknowledgements}
The first named author thanks Mihnea Popa for  conversations on the topic, which led to this project.  The first named author also thanks Jonathan Wise and David Rydh for helpful conversations about the geometry of stacks. The second named author thanks his advisor, Mircea Musta\c{t}\u{a}, for useful discussions and all the support provided, including many helpful comments on an earlier draft of this paper.   The authors are also grateful to the organizers of the Simons Collaborations on Moduli of Varieties Workshop at the University of Utah in November 2024, where their  work on this project began.

\section{Preliminaries}

\subsection{Terminology}

We work over $\mathbb C$.  
A \emph{variety} is an integral separated scheme of finite type over $\mathbb C$. 
An \emph{alteration} $X'\to X$ is a surjective projective generically
finite   morphism of schemes over $\mathbb C$.  
We use the definition of a \emph{Deligne--Mumford (DM) stack} in \cite[Def.~4.1]{LMB}. Note that this differs from the definition in \cite{stacks-project} in that there is the additional hypothesis in \cite[Def.~4.1]{LMB} that the diagonal be representable, separated, and quasi-compact.  We direct the reader to \cite[App.~B]{CMW18} for a discussion of the relationship among various definitions of DM stacks in the literature (see in particular \cite[Fig.~1]{CMW18}).  We emphasize that, with the definition of DM stack that we are using, a morphism from a scheme to a DM stack is schematic (representable by schemes); see e.g., \cite[Lem.~B.20 and Lem.~B.12]{CMW18}.

\subsection{Structure of DM stacks}\label{S:DM-intro}

  The general set-up will be a smooth proper (resp.~separated) integral DM stack $\mathcal X$ of finite type over $\mathbb C$ with coarse moduli space $\pi: \mathcal X\to X$, with the added assumption that the algebraic space $X$ be a projective (resp.~quasi-projective) variety. 
Recall that such a stack admits a finite flat 
morphism $q:V\to \mathcal X$ from a smooth projective  (resp.~quasi-projective) 
variety $V$  (\cite[Thm.~1]{KV04} and \cite[Thm.~4.4]{kresch09}, see also \cite[\href{https://stacks.math.columbia.edu/tag/03B6}{\S 03B6}]{stacks-project}); note that $q$ is schematic and projective.  
In this situation we have that $X$ is normal, $\mathbb Q$-factorial, with at worst klt singularities.  The morphism $\pi:\mathcal X\to X$ is flat over the smooth locus of $X$; flatness is an \'etale local property, and so it suffices to consider the case $\mathcal X=[U/G]$ for some smooth variety $U$ and a finite group $G$.  Then, from say \cite[Cor.~14.12]{GW20}, it suffices to show that $U\to U/G$ is flat over the smooth locus of the quotient, which follows from the miracle of flatness \cite[Thm.~23.1, p.179]{matsumura}.

For brevity,  will say that such a stack $\mathcal X$ is a global finite quotient stack if there is a smooth projective (resp.~quasi-projective) variety $V$ over $\mathbb C$ and a finite algebraic group $G$ over $\mathbb C$ acting on $V$ such that $\mathcal X\cong [V/G]$; note that this implies that $X=V/G$.  For context, recall that  $\mathcal X$ is a global finite quotient stack  if and only if there exists a smooth projective (resp.~quasi-projective) variety  $V'$ and a finite \emph{\'etale} morphism $q':V'\to \mathcal X$ (see the \emph{proof} of \cite[Thm.~(6.1)]{LMB}).

\begin{dfn}
A morphism $f:\mathcal{X}\rightarrow \mathcal{Y}$ of DM stacks is said to be projective if it has a relatively ample line bundle.
\end{dfn}

\subsection{Line bundles}

We also recall \cite[Lem. 2.1.2]{AGV08} that for any line bundle $\mathcal L$ on $\mathcal X$, there is a natural number $e$ such that $\mathcal L^{\otimes e}$ is the pull back   via $\pi$ of a line bundle $M$ on $X$.
We also note that for any Zariski open substack $i:\mathcal U\hookrightarrow  \mathcal X$ of codimension at least two, the pull back   morphism
\begin{equation}\label{E:PicXU}
i^*:\operatorname{Pic}(\mathcal X)\to \operatorname{Pic}(\mathcal U)
\end{equation}
is an isomorphism.  Indeed, any line bundle on $\mathcal U$ extends uniquely to a line bundle on $\mathcal X$, since $\mathcal X$ is smooth (and so in particular, $S_2$).   For instance, considering the line bundle as a line bundle on an associated groupoid \cite[\href{https://stacks.math.columbia.edu/tag/03LH}{\S 03LH}]{stacks-project}, the result follows from the fact that on a smooth scheme,  line bundles and morphisms of line bundles extend uniquely over codimension-$2$ loci.

\subsection{Determinants of torsion-free coherent sheaves}
For a torsion-free coherent sheaf $\mathcal F$ on $\mathcal X$ we can define a determinant:
\begin{equation}\label{E:det-def}
\det \mathcal F :=\left(\bigwedge ^{\operatorname{rk}\mathcal F} \mathcal F\right)^{\vee \vee}.
\end{equation}
A key observation is that $\mathcal F$ is locally free over a Zariski open substack $i:\mathcal U\hookrightarrow \mathcal X$ with complement of codimension at least two.  Indeed, this is a local question, and so we can reduce to the question of torsion-free sheaves on smooth varieties.  Considering local rings of codimension-$1$, one can see that $\mathcal F$ is locally free in codimension-$1$.  Consequently, if $(i^*)^{-1}:\operatorname{Pic}(\mathcal U)\to \operatorname{Pic}(\mathcal X)$ is the inverse of \eqref{E:PicXU}, then 
\begin{equation}\label{E:detVBU}
\det \mathcal F = (i^*)^{-1}(\bigwedge^{\operatorname{rk} \mathcal F} \mathcal F|_{\mathcal U});
\end{equation}
 i.e., the determinant of $\mathcal F$ is the unique extension of the determinant of the vector bundle $\mathcal F|_{\mathcal U}$ over $\mathcal U$ to a line bundle on all of $\mathcal X$.

If $\mathcal X=X$ is a smooth projective variety, then this definition of determinant  clearly agrees with the one in \cite{KMdet}.   Moreover, if $q:V\to \mathcal X$ is a finite flat morphism from a smooth projective variety $V$, then $\det (q^*\mathcal F)=q^*\det (\mathcal F)$.

We can use \eqref{E:detVBU} to show that if 
$$
0\to \mathcal F'\to \mathcal F\to \mathcal F''\to 0
$$
is a short exact sequence of torsion-free coherent sheaves on $\mathcal X$, then 
$$
\det \mathcal F \cong \det \mathcal F' \otimes \det \mathcal F''.
$$
Indeed, there is a common Zariski open substack $i:\mathcal U\hookrightarrow \mathcal X$ of codimension at least two such that all of the sheaves are locally free.  Then one can use that the result holds clearly for short exact sequences of vector bundles, together with \eqref{E:detVBU}.

\section{Positivity on smooth DM stacks}\label{S:positivityI}

The goal of this section is to clarify what we  mean by certain notions of positivity in the setting of smooth DM stacks. As our goal is to eventually obtain results about the geometry of the coarse moduli space of the stack, our definitions are made with this in mind.  Throughout this section, we will assume that $\mathcal X$ is a smooth proper DM stack over $\mathbb C$ with projective coarse moduli space $\pi:\mathcal X\to X$.  Recall from \S \ref{S:DM-intro} that such stacks admit a finite flat morphism
$$
q:V\to \mathcal X
$$
from a smooth projective variety $V$.

\subsection{Nef line bundles}

We define nef line bundles on our stacks $\mathcal X$ as follows:

\begin{dfn}[Nef line bundles]
We say that a line bundle $\mathcal L$ on $\mathcal X$ is \emph{nef} if some positive tensor power of $\mathcal L$ descends to a nef line bundle on the coarse moduli space $X$. 
\end{dfn}

\begin{rem}[Warning on \emph{ample} terminology] 
To avoid possible confusion with various natural definitions of \emph{ample} line bundles on stacks, we will \emph{not} refer to a  line bundle 
$\mathcal L$ on $\mathcal X$, such that some positive tensor power of $\mathcal L$ descends to an ample line bundle on the coarse moduli space $X$, as an ample line bundle on $\mathcal X$. 
 
\end{rem}

We have the following equivalent characterizations:

\begin{lem}\label{L:ample} For a line bundle $\mathcal L$ on $\mathcal X$, the following are equivalent:
\begin{enumerate}[label=(\arabic*)]
\item  $\mathcal L$ is nef (resp.~has a positive tensor power that descends to an ample line bundle on $X$).

\item \label{L:ampleV}
There exists a finite flat morphism $q:V\to \mathcal X$ from a smooth projective variety $V$ such that the pull back   $q^*\mathcal L$ to $V$ is nef (resp.~ample).

\item \label{L:ampleV'} There exists a  finite surjective morphism $q':V'\to \mathcal X$ from a smooth projective variety $V'$ such that the line bundle $q'^*\mathcal L$ is nef  (resp.~ample).

\item \label{L:ample_allV'} For any  finite surjective  morphism $q':V'\to \mathcal X$ from a smooth projective variety $V'$,  the line bundle $q'^*\mathcal L$ is nef  (resp.~ample).  
\end{enumerate}
\end{lem}

\begin{proof}
Since $\mathcal X$ admits a finite flat morphism $q:V\to \mathcal X$ from a smooth projective variety $V$, to prove the lemma,
it suffices to show that if $q':V'\to \mathcal X$ is a  finite  surjective  morphism from a smooth projective variety $V'$,  then $\mathcal L$ is nef  (resp.~has a positive tensor power that descends to an ample line bundle on $X$)  if and only if  $q'^*\mathcal L$ is nef (resp.~ample).  
To fix notation, let us assume that $\mathcal L^{\otimes e}\cong \pi^*M$  for some natural number $e$ and some line bundle $M$ on $X$.  It suffices to show that $M$ is nef (resp.~ample) if and only if $(\pi \circ q')^*M$ is nef (resp.~ample).  
 Note now use \cite[Lem.~1.2.13 and Cor.~1.2.28]{Laz04I} and \cite[Exa.~1.4.4]{Laz04I}. 
\end{proof}

\subsection{Big and pseudoeffective line bundles}\label{S:big-ps-e-lb}
We start by recalling a few facts about big and pseudoeffective line bundles on normal projective varieties.  
 To begin, recall  the numerical dimensions $\kappa_\sigma$ defined in \cite[p.6 and Rem. 2.6]{CHMS14} for line bundles on normal projective varieties.  The numerical dimension 
  $\kappa_\sigma$ is stable under pull back by alterations (\cite[Prop.~2.7]{CHMS14}). Recall that a line bundle $L'$ on a smooth projective variety $X'$ is big (resp.~pseudoeffective) if and only if $\kappa_\sigma(L')=\dim X'$  (resp.~$\geq 0$) \cite[p.7]{CHMS14} 
  (resp.~\cite[Def.~2.2.25]{Laz04I} and \cite[Prop.~2.8]{CHMS14}).  
   In fact, this implies that on a normal projective variety $X'$, then   $L'$  is big (resp.~pseudoeffective)  if and only if $\kappa_\sigma(L')=\dim X'$  (resp.~$\geq 0$); this is essentially \cite[Rem.~2.9]{CHMS14}.  More precisely if $\tau':X''\to X'$ is a resolution of singularities, 
  then we have that $L'$ on $X'$ is big (resp.~pseudoeffective) if and only if $\tau'^*L'$ is big (resp.~pseudoeffective), e.g., \cite[p.67]{KM98}, or \cite[p.148, Cor.~2.2.45, and Prop.~2.2.43]{Laz04I}; then use that $\kappa_\sigma$ is stable under pull back   by alterations.

We define big and pseudoeffective line bundles on our stacks $\mathcal X$ as follows:

\begin{dfn}[Big and pseudoeffective line bundles]
We say that a line bundle $\mathcal L$ on $\mathcal X$ is \emph{big}  (resp.~\emph{pseudoeffective})  if some positive tensor power of $\mathcal L$ descends to a big (resp.~pseudoeffective)  line bundle on the coarse moduli space $X$.  
\end{dfn}

We have the following equivalent characterizations:

\begin{lem}\label{L:big} For a line bundle $\mathcal L$ on $\mathcal X$, the following are equivalent:
\begin{enumerate}[label=(\arabic*)]
\item  $\mathcal L$ is big (resp.~pseudoeffective).

\item \label{L:bigV} There exists a finite flat morphism $q:V\to \mathcal X$ from a smooth projective variety $V$ such that the pull back   $q^*\mathcal L$ to $V$ is big (resp.~pseudoeffective).

\item \label{L:bigV'} There exists a surjective generically finite  morphism $q':V'\to \mathcal X$ from a smooth projective variety $V'$ such that the line bundle $q'^*\mathcal L$ is big (resp.~pseudoeffective).

\item \label{L:big_allV'} For any surjective generically finite morphism $q':V'\to \mathcal X$ from a smooth projective variety $V'$,  the line bundle $q'^*\mathcal L$ is big (resp.~pseudoeffective).

\end{enumerate}
\end{lem}

\begin{proof} We give the proof in the case of big line bundles; the case of pseudoeffective line bundles is the same.  
Since $\mathcal X$ admits a finite flat morphism $q:V\to \mathcal X$ from a smooth projective variety $V$, to prove the lemma, it suffices to show that if $q':V'\to \mathcal X$ is a surjective generically finite morphism from a smooth projective variety $V'$,  then $\mathcal L$ is big if and only if  $q'^*\mathcal L$ is big.  
To fix notation, let us assume that $\mathcal L^{\otimes e}\cong \pi^*M$ for some natural number $e$ and some line bundle $M$ on $X$.  
First, let us consider case that $M$ is big.  Then, using the invariance of $\kappa_\sigma$ under alterations we have that  $\dim V' = \dim X = \kappa_\sigma(M)=\kappa_\sigma ((\pi\circ q')^*M)=\kappa_\sigma (q'^*\mathcal L^{\otimes e})$, showing that $q'^*\mathcal L^{\otimes e}$ is big.
Conversely, suppose that $q'^*\mathcal L^{\otimes e}= (\pi\circ q')^*M$ is big.
Then we have $ \dim X= \dim V'=\kappa_\sigma((\pi\circ q')^*M)=\kappa_\sigma(M)$, showing that $M$ is big. 
\end{proof}

The following is representative of a number of similar results one has generalizing statements about line bundles on varieties to line bundles on DM stacks.

\begin{lem}\label{L:big+eff}
If $\mathcal L$ is big and $\mathcal M$ is pseudoeffective, then  $\mathcal L\otimes \mathcal M$ is big.
\end{lem}

\begin{proof}
This follows directly from the results above and the corresponding statements on schemes. 
\end{proof}

\subsection{Weak positivity on stacks} 

We start by reviewing the definition of weak positivity on smooth quasi-projective varieties following Viehweg.  We note that there are some minor variations in the definition given in the literature, depending on whether one works with smooth or normal varieties, and whether one works with locally free sheaves or coherent sheaves; see e.g., \cite[Def.~p.643]{Vhilb1}, \cite[Def.~2.11 and Variant 2.13]{Vmoduli}, and \cite[Def.~3.1]{Vkod2}.

We will use the formulation in \cite[Def.~3.1]{PS17}, for torsion-free coherent sheaves on smooth quasi-projective varieties.
Let $X$ be a smooth quasi-projective variety, let $U\subseteq X$ be an open
subvariety, and let $F$ be a torsion-free coherent sheaf on $X$. We say that $F$ is \emph{generated by global sections at each point of $U$} if the natural map $$H^0(Y,F)\otimes_{\mathbb C} \mathcal O_X\longrightarrow  F$$ is surjective when restricted to $U$. 
The sheaf  $F$ is \emph{weakly positive over $U$}   if for every integer $\alpha > 0$ and every ample line bundle $A$ on $ X$, there is an integer $\beta >0$ 
such that
$$
(\operatorname{Sym}^{\alpha \beta} F)^{\vee \vee}\otimes A^{\beta} 
$$
is generated by global sections at each  point of $U$. 
We say that $F$ is weakly
positive if such an open subset $U \ne \emptyset$ exists.

\begin{rem}
If $i_0:U_0\hookrightarrow X$ is the open subset over which the torsion-free coherent sheaf $F$ is locally free, then the complement of $U_0$ has codimension at least two.  For any natural number $\nu$ we have $(\operatorname{Sym}^{\nu} F)^{\vee \vee}\cong i_{0,*}(\operatorname{Sym}^\nu F|_{U_0})$, and for this reason, when making computations, one can frequently assume that $F$ is locally free.  
\end{rem}

\begin{rem}\label{R:torsion}
Occasionally, for the purpose of allowing for the pull back   of torsion-free sheaves along non-flat morphisms (cf.~\cite[\href{https://stacks.math.columbia.edu/tag/0AXV}{Lem.~0AXV}]{stacks-project}),  which may fail to be torsion-free, we will want to allow for coherent sheaves  $F$ that are not torsion-free. In  that case, following \cite[Variant 2.13]{Vmoduli}, we say that $F$ is weakly positive over $U$ if the torsion-free quotient $F/F_{\operatorname{torsion}}$ is locally free on $U$ and weakly positive over $U$. 
\end{rem}

We will want to use the following basic properties:

\begin{lem}[{\cite{Vmoduli,Vkod2, Vhilb1}}]\label{L:WP-prpty}
Let $F$ be a torsion-free coherent sheaf on a smooth quasi-projective variety $X$.
\begin{enumerate}[label=(\arabic*)]

\item \label{L:WP-prpty-alt} If $\tau: X'\to X$ is an alteration of smooth quasi-projective varieties, then, if $F$ is locally free over $U\subseteq X$, then  $F$ is weakly positive over $U$ if and only $\tau^*F$ is weakly positive over $\tau^{-1}(U)$ \cite[3.4 e)]{Vhilb1}.

\item \label{L:WP-prpty-surj} If $F\to  F'$ is a morphism of torsion-free coherent sheaves on $X$, which is surjective over an open subset $U\subseteq X$, then if $ F$ is weakly positive over $U$, so is  $F'$; \cite[Lem.~2.16(c)]{Vmoduli} or \cite[Lem.~1.4 1)]{Vkod1}.

\end{enumerate}
\end{lem}

\begin{proof} As the hypotheses in the references are a little hard to keep track of,  (although the differences in hypotheses are not particularly consequential in the end) we include a brief discussion here.   
\ref{L:WP-prpty-alt} This is stated as  \cite[3.4 e)]{Vhilb1}; however, note that in the definition of weakly positive used in that paper \cite[Def.~p.643]{Vhilb1}, it is required that the sheaf be locally free when restricted to $U$  (and this hypothesis is used in the proof of  \cite[3.4 e)]{Vhilb1} on p.652 to reduce to the locally free case).  
\ref{L:WP-prpty-surj} This is stated in  \cite[Lem.~2.16(c)]{Vmoduli} under the weaker hypothesis that $F$ and $F'$ just be coherent sheaves, but then under the stronger hypothesis that $F'$ be locally free over $U$.  This is stated in the formulation we have here, for torsion-free coherent sheaves, in  \cite[Lem.~1.4 1)]{Vkod1}.
\end{proof}

We now turn to the situation of stacks:

\begin{dfn}[Weakly positive]\label{D:PotWP}
    A torsion-free coherent sheaf $\mathcal F$ on ${\mathcal X}$ is \emph{weakly positive over a Zariski open substack $\mathcal {U}\subseteq \mathcal X$} if there exists a finite flat surjective morphism $q:V\to \mathcal X$ from a smooth projective variety so that $q^*\mathcal F$ is weakly positive over $q'^{-1}(\m{U})$.  We say that $\mathcal F$ is \emph{weakly positive} if such a cover and open substack $\mathcal U \ne \emptyset$ exist.
\end{dfn}

We have the following equivalent characterizations of weak positivity:

\begin{lem}\label{L:WPchar} For a torsion-free coherent sheaf  $\mathcal F$ on $\mathcal X$ that is locally-free over an open substack $\mathcal U\subseteq \mathcal X$, the following are equivalent:
\begin{enumerate}[label=(\arabic*)]
\item  \label{L:WPcharWP} $\mathcal F$ is weakly positive over $\mathcal U$.

\item \label{L:WPcharV}
For any finite flat morphism $q:V\to \mathcal X$ from a smooth projective variety $V$,  we have that  $q^*\mathcal F$ is weakly positive over $q^{-1}(\mathcal U)$.

\item \label{L:WPchar_allV'} For any surjective generically finite morphism $q':V'\to \mathcal X$ from a smooth projective variety $V'$,  we have that  $q'^*\mathcal F$ is weakly positive over $q'^{-1}(\mathcal U)$.

\end{enumerate}
\end{lem}

\begin{proof}
Since there exists a finite flat morphism $q:V\to \mathcal X$ from a smooth projective variety $V$, clearly, we have \ref{L:WPchar_allV'} $\implies$ \ref{L:WPcharV} $\implies$ \ref{L:WPcharWP}.
Let us now show that \ref{L:WPcharWP} $\implies$ \ref{L:WPchar_allV'}.  To this end,  assume that there exists a finite flat  morphism $q:V\to \mathcal X$ from a smooth projective variety $V$ such that $q^*\mathcal F$ is weakly positive.  
And now let  $q':V'\to \mathcal X$ be any surjective generically finite morphism from a smooth projective variety $V'$.  Taking a resolution of singularities of the fibered product $V\times_{X}V'$, we obtain a smooth projective variety $V''$ and  alterations $q_1:V''\to V$ and $q_2:V''\to V'$.  The result then follows from 
\Cref{L:WP-prpty}\ref{L:WP-prpty-alt}.  
\end{proof}

\begin{rem}\label{R:WPosLB}
Recall that a line bundle $L$ on a smooth projective variety $V$ is weakly positive (over some non-empty subset $U\subseteq V$) if and only if $L$ is pseudo-effective (for line bundles, the definition of weak positivity is equivalent to the condition that for any ample line bundle $A$ on $V$ and any $\epsilon>0$, one has $L+\epsilon A$ is in the cone generated by effective divisors).  Note that, perhaps somewhat confusingly, the line bundle $L$ is weakly positive \emph{over} $V$ if and only if it is  nef (e.g., \cite[Rem.~2.12.2]{Vmoduli}).  
It then follows from \Cref{L:WPchar} and \Cref{L:ample} that a  line bundle $\mathcal L$ on $\mathcal X$ is  weakly positive (over some non-empty open substack $\mathcal U\subseteq \mathcal X$)  if and only if it is pseudo-effective.  Similarly, the line bundle is weakly positive \emph{over $\mathcal X$} if and only if it is nef.
\end{rem}

\begin{rem}
There is a notion of weak positivity for normal projective varieties \cite[Def.~p.643]{Vhilb1}.   
Suppose that  $\mathcal F$ descends to $F$ on $X$, that  $U$ is an open subset such that the singular locus of $U$ is proper, and that $F$ is locally free on $U$.  Then $F$ is weakly positive over $U$ if and only if $\mathcal F$ is weakly positive over $\pi^{-1}(U)\subseteq \mathcal X$ \cite[3.4 e) and Assumptions 3.1]{Vhilb1}. In particular, by restricting if necessary  to an open set $U$ that is smooth and where $F$ is locally free, one has that  $F$ is weakly positive on $X$ if and only if $\mathcal F$ is weakly positive on $\mathcal X$.  
\end{rem}

We also have:

\begin{lem}\label{L:surjWP}
If  $\mathcal F\to \mathcal F'$ is a morphism of torsion-free coherent sheaves on $\mathcal X$, which is surjective over $\mathcal U$, then if $\mathcal F$ is weakly positive over $\mathcal U$, so is $\mathcal F'$.
\end{lem}

\begin{proof}
We consider the pullback $q^*\mathcal F\to q^*\mathcal F'$, which is surjective over $q^{-1}(\mathcal U)$.  (As $q$ is flat, we have that $q^*\mathcal F$ and $q^*\mathcal F'$ are torsion-free.)  Then, from 
\Cref{L:WP-prpty}\ref{L:WP-prpty-surj},  we have that $q^*\mathcal F'$ is weakly positive over $q^{-1}(\mathcal U)$.  Therefore, $\mathcal F'$ is weakly positive over $\mathcal U$. 
\end{proof}

\subsection{Big sheaves on stacks}

We start by reviewing the definition of big sheaves on smooth quasi-projective varieties following Viehweg \cite[Def.~2.7]{Vmoduli} and \cite[Lem.~3.6]{Vkod2} (see also \cite[Def.~3.1]{PS17}).
Let $X$ be a smooth projective variety and let $F$ be a non-zero torsion-free coherent sheaf on $X$. 
Then $F$ is big if there exists an ample line bundle $A$ on $X$, a natural number $\nu$, and an inclusion
\begin{equation}\label{E:bigF-defSch}
\bigoplus ^{\operatorname{rank}\operatorname{Sym}^\nu F}A \hookrightarrow (\operatorname{Sym}^\nu F)^{\vee \vee}.
\end{equation}
This is equivalent to saying that for every line bundle $L$ on $X$, there exists an integer $\gamma>0$ such that $(\operatorname{Sym}^\gamma F)^{\vee \vee} \otimes L^{-1}$ is weakly positive \cite[Lem.~3.6]{Vkod2}.

\begin{rem}\label{R:torsionBig} As in the case of weak positivity (\Cref{R:torsion}), 
occasionally, for the purpose of allowing for the pull back   of torsion-free sheaves along non-flat morphisms (cf.~\cite[\href{https://stacks.math.columbia.edu/tag/0AXV}{Lem.~0AXV}]{stacks-project}),  which may fail to be torsion-free, we will want to allow for coherent sheaves  $ F$ that are not torsion-free.  For such sheaves, we use the same definition of big; note that this is equivalent to the torsion-free quotient $F/F_{\operatorname{torsion}}$ being big.  
\end{rem}

We will want to use the following basic properties (e.g., \cite[Lem.~3.2]{PS17}):

\begin{lem}

\label{L:big-prpty}
Let $F$ be a torsion-free coherent sheaf on a smooth projective variety $X$.
\begin{enumerate}[label=(\arabic*)]

\item 

\label{L:big-prpty-ff} If $\tau: X'\to X$ is a finite surjective morphism of 

of smooth projective varieties, then, if $F$ is big, so is  $\tau^*F$.

\item \label{L:big-prpty-surj} If $F\to  F'$ is a morphism of torsion-free coherent sheaves on $X$, which is surjective over an open subset $U\subseteq X$, then if $ F$ is big, so is $F'$.

\item \label{L:big-prpty-wpb} If $F$ is weakly positive and $B$ is a big line bundle on $X$, then $F\otimes B$ is big.

\item \label{L:big-prpty-det} If $F$ is  big, then $\det F$ is a big line bundle.

\end{enumerate}
\end{lem}

\begin{proof} We provide a proof for the convenience of the reader. 
\ref{L:big-prpty-ff} 
As $\mathcal F$ is torsion-free, we can compute on the locus where $\mathcal F$ is locally free; in particular $\tau^*((\operatorname{Sym}^\nu F)^{\vee\vee})\cong (\operatorname{Sym}^\nu \tau^* F)^{\vee\vee}$.  
As $\tau$ is generically flat, the kernel of the pullback of \eqref{E:bigF-defSch} would have to be torsion; but the pull back   of a locally free sheaf is locally free, and therefore torsion-free, so that the kernel, being contained in a torsion-free sheaf would have to be zero.  
One then concludes using the fact that the pull back   of an ample line bundle by a finite morphism is ample. 
 \ref{L:big-prpty-surj} Use that the morphism $F\to F'$ induces a morphism $(\operatorname{Sym}^\gamma F)^{\vee \vee} \otimes L^{-1}\to (\operatorname{Sym}^\gamma F')^{\vee \vee} \otimes L^{-1}$, surjective over $U$; then use \Cref{L:WP-prpty}\ref{L:WP-prpty-surj}. \ref{L:big-prpty-wpb}  follows from the definitions and is left to the reader. \ref{L:big-prpty-det}  This is essentially the same as \cite[Lem.~1.4 6)]{Vkod1}.   Assume that $F$ is big, so that there is an inclusion as in \eqref{E:bigF-defSch}.   Let $Q$ be the quotient, which is torsion.  We then have
\begin{equation}\label{E:L:big-prpty-pf}
(\det F)^a=\det (\operatorname{Sym}^\nu F)^{\vee\vee} = A^{\operatorname{rank}\operatorname{Sym}^\nu F} \otimes \det Q,
\end{equation}
where $a\cdot \operatorname{rank}F= \nu \cdot \operatorname{rank}\operatorname{Sym}^\nu F$ (e.g., \cite[Lem.~1.5]{Vkod1}), and we are using that, since $F$ is torsion-free, the determinant can be computed on the locus where $F$ is locally free.  As the determinant of a torsion sheaf is effective, 
we have that $\det F$ is big.
\end{proof}

We now turn to the situation of stacks:

\begin{dfn}[Big sheaves] 
A torsion-free coherent sheaf $\mathcal F$ on $\mathcal X$ is \emph{big}
if there exists a finite flat  morphism $q:V\to \mathcal X$ from a smooth projective variety $V$ such that 
$q^*\mathcal F$ is  big.
\end{dfn}

\begin{rem}\label{R:bigShdesc}
There is a notion of big sheaves for normal projective varieties \cite[Def.~2.7 b)]{Vmoduli}.   
If $\mathcal F$ descends to $F$ on $X$, then  if $F$ is big, so is   $\mathcal F$; 
the argument is the same as \Cref{L:big-prpty}\ref{L:big-prpty-ff}, using that the finite morphism $\pi:\mathcal X\to X$ is generically flat (in particular, flat over the smooth locus of $X$; see \S \ref{S:DM-intro}).
\end{rem}

The following basic properties follow from \Cref{L:big-prpty}:

\begin{lem}\label{L:pbig-prpty}
Let $\mathcal F$ be a torsion-free coherent sheaf on   $\mathcal X$.
\begin{enumerate}[label=(\arabic*)]

\item \label{L:pbig-prpty-surj} If $\mathcal F\to  \mathcal F'$ is a morphism of torsion-free coherent sheaves on $\mathcal X$, which is surjective over a Zariski open substack $\mathcal U\subseteq \mathcal X$, then if $ \mathcal F$ is   big, so is $\mathcal F'$.

\item \label{L:pbig-prpty-wpb} If $\mathcal F$ is  weakly positive and $\mathcal B$ is a big line bundle on $\mathcal X$, then $\mathcal F\otimes \mathcal B$ is   big.

\item \label{L:pbig-prpty-det} If $\mathcal F$ is   big, then $\det \mathcal F$ (see \eqref{E:det-def}) is a big line bundle.

\end{enumerate}
\end{lem}

\begin{proof}
 \ref{L:pbig-prpty-surj} 
 Let $q:V\to \mathcal X$ be a finite flat morphism from a smooth projective variety $V$ such that $q^*\mathcal F$ is big.  By restricting to a smaller open subset, we may assume that $\mathcal F$ and $\mathcal F'$ are locally free over $\mathcal U$.  The morphism $q^*\mathcal F\to q^*\mathcal F'$ induces a surjection on torsion-free quotients over $\mathcal U$.  As $q^*\mathcal F$ is big by assumption, we have that $q^*\mathcal F'$ is big from \Cref{L:big-prpty}\ref{L:big-prpty-surj}. Therefore $\mathcal F'$ is  big.
 
 \ref{L:pbig-prpty-wpb}  Let $q:V\to \mathcal X$ be a finite flat morphism from a smooth projective variety such  that $q^*\mathcal F$ is weakly positive.  
 We have that $q^*\mathcal B$ is big (\Cref{L:big}\ref{L:bigV'}).  It follows that $q^*(\mathcal F\otimes \mathcal B)=q^*\mathcal F\otimes q^*\mathcal B$ is big (\Cref{L:big-prpty}\ref{L:big-prpty-wpb}), so that $\mathcal F\otimes \mathcal B$ is  big.
 
 \ref{L:pbig-prpty-det} 
 Let $q:V\to \mathcal X$ be a finite flat morphism from a smooth projective variety $V$ such that $q^*\mathcal F$ is big. From \Cref{L:big-prpty}\ref{L:big-prpty-det} we have that $\det q^*\mathcal F$ is big.   Let $\mathcal U\subseteq \mathcal X$ be the open substack over which $\mathcal F$ is locally free, which has complement of codimension at least $2$ in $\mathcal X$.  As $q$ is finite, we also have that the complement of $q^{-1}(\mathcal U)$ in $V$ is  of codimension at least $2$.  Therefore, since for locally free sheaves the pull back   and determinant commute, we have that $
q^* \det \mathcal F =   \det q^*\mathcal F
$ is big, which implies that $\det \mathcal F$ is big (\Cref{L:big}).
\end{proof}

\section{$D$-modules and Hodge modules DM stacks} \label{S:DM-HM-stacks}

We now discuss $D$-modules and Hodge modules on DM stacks.   
In general, if one wants to obtain a 6-functor formalism, there are more sophisticated treatments, e.g., \cite{tubach_2024} or \cite{paulin13} for $D$-modules; however, for our purposes here, where we only need to pull back   pure polarizable Hodge modules along non-characteristic morphisms, the more elementary approach we take here, closely following say the treatment of $D$-modules on stacks in \cite{BDstacks}, suffices.  
In this section we work throughout with a smooth separated DM stack $\mathcal X$, locally of finite type over $\mathbb C$.

\subsection{$D$-modules on DM stacks}
We start with the topic of $D$-modules.  There seems to be the standard, albeit brief, reference \cite{BDstacks}, which deals with the more complicated situation of smooth Artin stacks.  The situation for smooth DM stacks is easier, and does not require much additional input beyond the case for smooth varieties.   
We will briefly explain how to extend the presentation in  \cite{HTT08} to the case of DM stacks; we also point the reader to  \cite{BDstacks, Gomez10} for more details, and to \cite{paulin13} for a more sophisticated approach geared towards the $6$-functor formalism.

\subsubsection{Differential operators}

There is a well-defined \emph{(algebraic) tangent sheaf} $\mathcal T_{\mathcal X}$ on the \'etale site of $\mathcal X$.  By definition, for an \'etale morphism $U\to \mathcal X$, this  restricts to sheaf  ${\mathcal T}_{U}$ on the \'etale site $U_{\text{\'et}}$ given by  sections of the tangent bundle $T_U$ to $U$.  In other words, ${\mathcal T}_{U}=\mathcal Der_{\mathbb C_U}(\mathcal O_U)$, i.e., for an \'etale morphism $U'\to U$, 
\begin{align*}
{\mathcal T}_{U}(U')&=\mathcal Der_{\mathbb C}(\mathcal O_{U'}(U'))\\
& = \{\theta \in \mathcal End_{\mathbb C}(\mathcal O_{U'}(U')) : \theta(fg)=\theta(f)g+ f\theta(g)\}.
\end{align*}
The \'etale sheaves $\mathcal T_U$ determine the \'etale sheaf $\mathcal T_{\mathcal X}$ on $\mathcal X$. 
Clearly $\mathcal T_{\mathcal X}$ is a locally free $\mathcal O_{\mathcal X}$-module of rank $\dim \mathcal X$.  We denote by $\Omega^1_{\mathcal X}:=\mathcal T_{\mathcal X}^\vee$ the dual locally free sheaf, i.e., the sheaf of holomorphic $1$-forms, or the \emph{(algebraic) cotangent sheaf}.  
We define the \emph{tangent stack} as 
$$
T{\mathcal X}:=\underline{\operatorname{Spec}}_{\mathcal X}\operatorname{Sym}^\bullet \Omega^1_{\mathcal X}.
$$
This is a smooth DM stack; if $P:U\to \mathcal X$ is an \'etale presentation, then the induced morphism $T_{U}\to T_{\mathcal X}$ is an \'etale presentation of the tangent stack.  We can also consider the \emph{cotangent stack}
$$
T^\vee{\mathcal X}:=\underline{\operatorname{Spec}}_{\mathcal X}\operatorname{Sym}^\bullet {\mathcal T}_{\mathcal X},
$$
which is also a smooth DM stack.

We have a natural inclusion of \'etale sheaves of   ${\mathcal T}_{\mathcal X}$ into $\mathcal End_{\mathbb C_{\mathcal X}}(\mathcal O_{\mathcal X})$, where $\mathbb C_{\mathcal X}$ is the locally constant sheaf on the \'etale site of  $\mathcal X$, and 
we include $\mathcal O_{\mathcal X}$ into $\mathcal End_{\mathbb C_{\mathcal X}}(\mathcal O_{\mathcal X})$ in the natural way, as well,  via multiplication of functions.  We then define $D_{\mathcal X}$, the \emph{sheaf of differential opperators on $\mathcal X$},  as the $\mathbb C_{\mathcal X}$-subalgebra of $\mathcal End_{\mathbb C_{\mathcal X}}(\mathcal O_X)$ generated by $\mathcal O_{\mathcal X}$ and ${\mathcal T}_{\mathcal X}$; we view this as a sheaf of (non-commutative) rings on the \'etale site $\mathcal X_{\text{\'et}}$.

\subsubsection{Order filtration}
Recall that for any smooth affine variety $U$ with local coordinate system $\{x_i,\partial_i\}$ we define the order filtration $F$ of $D_U$ by $F_\ell D_U=\sum_{|\alpha|\le \ell}\mathcal O_U\partial^\alpha_x$, for $\ell\in \mathbb N$ and $|\alpha|=\sum \alpha_i$, with $F_pD_U=0$ for convenience for $p<0$.   
This defines a filtration, which is independent of the choice of local coordinate system (see \cite[\S 1.1]{HTT08}),  and is clearly stable under pull back   by \'etale morphisms.  
 Then for any \'etale morphism $U\to \mathcal X$ we define 
$$
(F_\ell D_{\mathcal X})(U):=\{P\in D_{\mathcal X}(U) : \operatorname{res}^U_{U'}P\in F_\ell D_{\mathcal X}(U')  \text { for any affine \'etale } U'\to U\}
$$
where $\operatorname{res}_{U'}^U:D_{\mathcal X}(U)\to D_{\mathcal X}(U')$ is the natural map associated to the \'etale cover $U'\to U$ coming from viewing $D_{\mathcal X}$ as a sheaf on $\mathcal X_{\text{\'et}}$.

\subsubsection{$D$-modules on stacks}

With the \'etale sheaf of rings $D_{\mathcal X}$ on the \'etale site $\mathcal X_{\text{\'et}}$, we make the following definition of a $D$-module on a stack:

\begin{dfn}[$D$-module] \label{D:DmodXX}
Let $\mathcal X$ be a smooth separated DM stack locally of finite type over $\mathbb C$.  
A \emph{$D$-module on $\mathcal X$} is a  sheaf of (left) $D_{\mathcal X}$-modules on the \'etale site $\mathcal X_{\text{\'et}}$, i.e., a sheaf of   $D_{\mathcal X}$-modules on the ringed site  $(\mathcal X_{\text{\'et}},D_{\mathcal X})$. Morphisms of $D$-modules on $\mathcal X$ are given by morphisms a sheaves of $D_{\mathcal X}$-modules on the ringed site.  
\end{dfn}

 We will often use the terminology \emph{$D_{\mathcal X}$-module} for a $D$-module on $\mathcal X$.  
  
  \begin{rem}
  As $D$-modules 
    satisfy \'etale descent for smooth varieties (see \cite[Prop. 19.4.7 (b)]{KashiwaraSchapira06}), the above definition generalizes the usual notion of $D$-modules on smooth varieties.  
    In other words, if $X$ is a smooth variety and $\mathcal M$ is a $D_{X}$-module in the sense if \Cref{D:DmodXX} (i.e., a sheaf on the \'etale site $X_{\text{\'et}}$), then there is a $D$-module $M$ on $X$ in the usual sense of say \cite[p.17]{HTT08} with associated \'etale sheaf $\mathcal M$.  We will frequently interchange these two perspectives on $D$-modules when working on varieties.
\end{rem}

We denote by $\operatorname{Mod}(D_{\mathcal X})$ the category of (left) $D_{\mathcal X}$-modules.  
Note that, as $D_{\mathcal X}$ is a locally free $\mathcal O_{\mathcal X}$-module, we can view $D_{\mathcal X}$-modules as modules on the ringed site $(\mathcal X_{\text{\'et}},\mathcal O_{\mathcal X})$.  We make the following local definition of (quasi-)coherent $D$-modules:

\begin{dfn}[(Quasi-)coherent $D$-modules on DM stacks]\label{D:DMcoh}
Let $\mathcal X$ be a smooth separated DM stack locally of finite type over $\mathbb C$.  A   $D$-module $\mathcal M$ on $\mathcal X$ is \emph{quasi-coherent (resp.~coherent)} if for every \'etale morphism $U\to \mathcal X$ from a smooth variety $U$ one has that the  induced $D$-module $\mathcal M_U$ is a quasi-coherent (resp.~coherent) $D$-module. 
\end{dfn}

\begin{rem}
Connecting with the standard definition on varieties (e.g., \cite[Not.~1.4.1]{HTT08}), note that from the definition above, a $D_{\mathcal X}$-module is quasi-coherent if and only if it is quasi-coherent as an $\mathcal O_{\mathcal X}$-module.  Similarly, a  $D_{\mathcal X}$-module is coherent if and only if it is coherent as a (left) $D_{\mathcal X}$-module.
\end{rem}

  We denote by 
$\operatorname {Mod}_{qc}(D_{\mathcal X})$ (resp.~$\operatorname{Mod}_c(D_{\mathcal X})$) the full sub-category of $\operatorname{Mod}(D_{\mathcal X})$ consisting of quasi-coherent (resp.~coherent) $D_{\mathcal X}$-modules.

\subsubsection{Descent data}\label{S:descentData}
In the usual way, if $\mathcal M$ is $D_{\mathcal X}$-module and $U\to \mathcal X$ is an \'etale morphism from a smooth variety, we define $\mathcal M_U$ to be the restriction of $\mathcal M$ to the \'etale site of $U$ via the given morphism.  
The definition of $\mathcal M$ implies that for any
 commutative diagram
 \begin{equation}\label{E:DmodDD1}
\xymatrix{
U'\ar[rr]^{\phi}\ar[rd]&&U \ar[ld]\\
&\mathcal X&
}
\end{equation}
of \'etale morphisms with $U'$ and $U$ smooth varieties,
we have a natural isomorphism 
 \begin{equation}\label{E:DmodDD2}
 \theta_{\phi}: \phi^*\mathcal M_U\stackrel{\sim}{\longrightarrow} \mathcal M_{U'}
\end{equation}
and moreover, given a 
commutative diagram
 \begin{equation}\label{E:DmodDD3}
\xymatrix{
U''\ar[r]^{\phi'}\ar[rd]&U'\ar[r]^\phi\ar[d]&U \ar[ld]\\
&\mathcal X&
}
\end{equation}
of \'etale morphisms with $U''$, $U'$,  and $U$ smooth varieties, we have an equality
 \begin{equation}\label{E:DmodDD4}
\theta_{\phi\circ \phi'} = \theta_{\phi'}\circ\phi'^*\theta_\phi.
\end{equation}
In addition, in the usual way, the data above determines a $D_{\mathcal X}$-module.  In other words, if for every \'etale morphism $U\to \mathcal X$ one has a $D$-module $M_U$ on $U$, as well as for every diagram \eqref{E:DmodDD1} isomorphisms $
 \theta_{\phi}: \phi^* M_U\stackrel{\sim}{\longrightarrow} \mathcal M_{U'}
$ satisfying \eqref{E:DmodDD4} for every diagram \eqref{E:DmodDD3}, then there is a unique $D_{\mathcal X}$-module $\mathcal M$ such that $\mathcal M_U=M_U$ for all $U\to \mathcal X$ \'etale.

\begin{rem}[The various notions of pull back   all agree]\label{R:AllPBagree}
As we defined $D$-modules as sheaves of modules on the ringed site $(\mathcal X_{\text{\'et}},D_{\mathcal X})$, the pull back   $\phi^*$  
 in 
 \eqref{E:DmodDD2} is by definition the functor 
 $\phi^*:\mathsf {Mod}(D_U)\to\mathsf {Mod}(D_{U'})$, 
$M_U\mapsto 
 D_{U'}\otimes_{\phi^{-1}D_U}\phi^{-1}M_U$.  
However, as  we are working with $D$-modules, there are a number of other pull back   functors that would make sense in this situation; as the morphism $\phi$ in \eqref{E:DmodDD1} is \'etale, all of these pull back   functors agree.
Indeed, note first that since $\phi$ is \'etale, the pull back   $\phi^*$ defined above at the level of ringed spaces  agrees with the definition of $\phi^*:\mathsf {Mod}(D_U)\to\mathsf {Mod}(D_{U'})$ in \cite[p.21]{HTT08}, where $\phi^*M_U= \mathcal O_{U'}\otimes_{\phi^{-1}\mathcal O_U}\phi^{-1}M_U$, i.e., where the pull back   is taken at the level of $\mathcal O$-modules. 
Indeed,  with $\phi$ \'etale, we have that $D_{U'}\cong \mathcal O_{U'}\otimes_{\phi^{-1}\mathcal O_U}\phi^{-1}D_U$, so that $D_{U'}\otimes_{\phi^{-1}D_{U}}  \phi^{-1}M_U\cong  (\mathcal O_{U'}\otimes_{\phi^{-1}\mathcal O_U}\phi^{-1}D_U)\otimes_{\phi^{-1}D_{U}}  \phi^{-1}M_U \cong  \mathcal O_{U'}\otimes_{\phi^{-1}\mathcal O_U}\phi^{-1}M_U$, so that the two definitions agree. 
  We   also have $\phi^*M_U=L\phi^*M_U$, where $L\phi^*:D^b(D_U)\to D^b(D_{U'})$ is the functor $M^\bullet\mapsto D_{U'\to U} \otimes^L_{\phi^{-1}D_U}\phi^{-1}M^\bullet$ (e.g., \cite[p.32]{HTT08}).   Indeed, we have $L\phi^* M_U=D_{U'\to U} \otimes^L_{\phi^{-1}D_U}\phi^{-1}M_U\cong  \mathcal O_{U'}\otimes^L_{\phi^{-1}\mathcal O_U}\phi^{-1}M_U\cong  \mathcal O_{U'}\otimes_{\phi^{-1}\mathcal O_U}\phi^{-1}M_U\cong \phi^*M_U$, where 
the first isomorphism comes from  \cite[Prop.~1.5.8]{HTT08}, and the second  isomorphism comes from the fact that $\phi$ is flat.
This then implies that  $\phi^*M_U=\phi^\dagger M_U$, where $\phi^\dagger M_U  := L\phi^*M_U[\dim U'-\dim U]$ (e.g., \cite[p.33]{HTT08}).
  Finally, we also have that $\phi^\star M_U=\phi^*M_U$, where $\phi^\star M_U:=\mathbb D_{U'}\phi^\dagger \mathbb D_U$ (e.g., \cite[p.91]{HTT08}); this follows from \cite[Thm.~2.7.1(ii) and Thm.~2.6.5(ii)]{HTT08}.
\end{rem}

We will often use the following concrete interpretation of a $D$-module on $\mathcal X$.  Given a presentation $p:U\to \mathcal X$, then to give a $D$-module on $\mathcal X$ is equivalent to giving a $D$-module $M$ on $U$ with descent data.
Recall that from the presentation of $\mathcal X$, we have a fibered product diagram
\begin{equation}\label{E:DescentData}
\xymatrix@C=.5em@R=1em{
&U_i\times_{\mathcal X}U_j\times_{\mathcal X}U_k \ar[rr]^{pr_{jk}} \ar[ld]_<>(0.6){pr_{ij}} \ar@{-}[d]^<>(0.5){pr_{ik}}&&U_j\times_{\mathcal X}U_k \ar[dd]^{pr_k} \ar[ld]_<>(0.6){pr_{j}}\\
U_i\times_{\mathcal X}U_j\ar[dd]^{pr_i} \ar[rr]^<>(0.25){pr_j}&\ar[d]&U_j\ar[dd]^<>(0.25)p&\\
&U_i\times_{\mathcal X}U_k \ar@{-}[r]^<>(0.5){pr_k} \ar[ld]_{pr_i}&\ar[r]&U_k \ar[ld]^p\\
U_i\ar[rr]_p&&\mathcal X&\\
}
\end{equation}
where $U_i=U_j=U_k=U$ and $p_i=p_j=p_k=p$ are given labels for clarity (or one can consider indexing the components of the cover). The morphisms are the indicated projections, and are all \'etale by base change.  As $p:U\to \mathcal X$ is schematic, all of the spaces in the diagram above are smooth schemes except possibly $\mathcal X$.  There is a well-defined pull back   for  $D$-modules  under \'etale morphisms of varieties (see also \Cref{R:AllPBagree});  to say that $M$ on $U$ has descent data is to say that for every ordered pair  $(\alpha,\beta)$ in $\{(i,j), (i,k), (j,k)\}$, one is given isomorphisms $$\theta_{\alpha\beta}: pr_\beta^*M\stackrel{\sim}{\to}pr_\alpha^*M$$ on $U_\alpha \times_{\mathcal X}U_\beta$ such that 
$$
(pr_{ij}^*\theta_{ij})\circ (pr_{jk}^*)\theta_{jk} = pr_{ik}^*\theta_{ik},
$$
where here the equality is up to the canonical isomorphism of the pull back   of $M$ through the various arrows in the diagram to $U_i\times_{\mathcal X}U_j\times_{\mathcal X}U_k$.  Morphisms are given in the obvious way by morphisms of objects on the $U_i$ that are compatible via pull back   in the diagram above.

\begin{rem}
As being a quasi-coherent $\mathcal O_{\mathcal X}$-module is a local condition, one has that a $D_{\mathcal X}$-module $\mathcal M$ is quasi-coherent if and only if $\mathcal M_U$ is a quasi-coherent $D_U$-module for every \'etale $U\to \mathcal X$ from a smooth variety $U$.  
\end{rem}

\subsubsection{Filtered $D$-modules}
We make the following definition of a filtered $D$-module motivated by the standard definition on varieties (e.g., \cite[p.57]{HTT08}):

\begin{dfn}[Filtered $D$-module] A \emph{filtered $D_{\mathcal X}$-module} $(\mathcal M,F_\bullet)$ is a pair consisting of a quasi-coherent $D_{\mathcal X}$-module  $\mathcal M$ as well as an increasing   filtration $F_i\mathcal M$ ($i \in \mathbb  Z$)  of $\mathcal M$ by coherent $\mathcal O_{\mathcal X}$-submodules  satisfying the conditions: (a) $F_i\mathcal M\subseteq F_{i+1}\mathcal M$, (b) $F_i\mathcal M=0$ for $i\ll 0$, (c) $\mathcal M=\bigcup_{i\in \mathbb Z} \mathcal M_i$, and (d) $(F_jD_{\mathcal X})(F_i\mathcal M)\subseteq F_{i+j}\mathcal M$.  
\end{dfn}

\begin{dfn}[Strict morphism]
A morphism $f:(\mathcal M,F_\bullet) \to (\mathcal N,F_\bullet)$ of filtered $D_{\mathcal X}$-modules is a morphism $f:\mathcal M\to \mathcal N$ of $D_{\mathcal X}$-modules such that $f(F_i\mathcal M)\subseteq F_i\mathcal N$ for all $i$; the morphism is \emph{strict} if it satisfies $f(F_i\mathcal M)= f(\mathcal M)\cap F_i\mathcal N$ for all $i$.  
\end{dfn}

Note that for any \'etale morphism $U\to \mathcal X$ from a smooth variety $U$, there is an induced filtered $D$-module $(\mathcal M_U,F_\bullet)$ obtained by pull back   to $U$ (see \Cref{R:AllPBagree} for a discussion of the identifications of various pull back   morphisms).   
For a filtered quasi-coherent $D_{\mathcal X}$-module $(\mathcal M,F^\bullet)$, we say the filtration is \emph{good} if for every \'etale morphism $U\to \mathcal X$ from a smooth variety $U$ one has that for the induced filtered $D$-module $(\mathcal M_U,F_\bullet)$ the filtration is good  \cite[Prop.~2.1.1 and Def.~2.1.2]{HTT08}. 

We now show that our local definition of coherent $D$-modules (\Cref{D:DMcoh}) implies the existence of a good (global) filtration (\emph{a priori}, from the definition, one only has good filtrations locally).

\begin{pro}\label{Existence-of-good-filtr-on-stack}
A $D_{\mathcal X}$-module $\mathcal M$ is coherent if and only if there exists a good filtration $F_\bullet \mathcal M$ on $\mathcal M$.
\end{pro}

\begin{proof} The implication ($\impliedby$) is clear.  So 
let $\mathcal{M}$ be a coherent $D_{\mathcal{X}}$-module on  $\mathcal{X}$. From the definitions being local, one has immediately from  \cite[Prop.~1.4.9(ii)]{HTT08} that $D_{\mathcal X}$ is a quasi-coherent $\mathcal O_{\mathcal X}$-module, locally finitely generated as a $D_{\mathcal X}$-module.  
    By \cite[Prop. 15.4]{LMB}, as $\mathcal M$ is a quasi-coherent $\mathcal{O}_{\mathcal{X}}$-module on $\mathcal{X}$, it is the colimit of its coherent $\mathcal O_{\mathcal X}$-subsheaves $\mathcal{N}_{\lambda}$. Consider $\mathcal{M}_{\lambda}:=\operatorname{Im}(D_{\mathcal{X}}\otimes_{\mathcal O_{\mathcal X}} \mathcal{N}_{\lambda} \to \mathcal{M})$. 
    The collection of  $D_{\mathcal X}$-submodules  $\mathcal M_\lambda \subseteq \mathcal{M}$ satisfies the ascending chain condition (ACC); indeed,  we can pull back to an \'etale cover $p:U\to \mathcal{X}$ and there ACC holds (e.g., \cite[Prop. 1.4.6]{HTT08}).
       So, there exists $\mathcal{N}_{\lambda}$, a  coherent $\mathcal{O}_{\mathcal{X}}$-module, for which $\mathcal{M}_{\lambda}=\mathcal{M}$. Then we obtain a filtration $F_{\bullet}$ on $\mathcal{M}$ by $F_k\mathcal{M}:=\operatorname{Im}(F_k D_{\mathcal{X}}\otimes_{\mathcal O_{\mathcal X}} \mathcal{N}_{\lambda}\to \mathcal{M}).$ It is an exhaustive filtration by coherent $\mathcal{O}_{\mathcal{X}}$-modules, it satisfies $F_k\mathcal{M}=0$ for $k<0$, and $F_iD_{\mathcal{X}}\cdot F_k\mathcal{M}=F_{k+i}\mathcal{M}$, for all $i\geq 0, k\geq 0$, and so we obtain a good filtration on $\mathcal M$.
\end{proof}

We now turn our attention to regular and holonomic $D$-modules.  We  note that  $D$-modules (resp.~filtered $D$-modules, resp.~quasi-coherent $D$-modules, resp.~coherent $D$-modules, resp.~regular $D$-modules, resp.~holonomic $D$-modules)  satisfy \'etale descent for smooth varieties (see e.g.,  \cite[Prop. 2.2 and Thm. 2.3]{achar_2013} and the references therein). Having already addressed the case of (quasi-)coherent $D$-modules, we use  this to motivate the following local definition for regular and holonomic $D$-modules on DM stacks:

\begin{dfn}[Regular and holonomic $D$-modules on DM stacks]\label{D:RHDM-stack}
Let $\mathcal X$ be a smooth separated DM stack locally of finite type over $\mathbb C$.  A   $D$-module $\mathcal M$ on $\mathcal X$ is \emph{regular (resp.~holonomic)} if for every \'etale morphism $U\to \mathcal X$ from a smooth variety $U$ one has that the  induced $D$-module $\mathcal M_U$ is a regular (resp.~holonomic) $D$-module.
\end{dfn}

\begin{rem}
Either from the descent results mentioned above, or from the more elementary fact that  regularity (resp.~holonomicity) for $D$-modules on varieties can be checked \'etale locally (see e.g., \cite[Def.~5.2.10 and Lem.~5.1.23]{HTT08} and \cite[Def.~2.3.6]{HTT08}) one sees that \Cref{D:RHDM-stack} agrees with the standard definition of regularity (resp.~holonomicity) when $\mathcal X$ is a variety. 
\end{rem}

We denote by $\operatorname{ Mod}_{c}(D_{\mathcal X})$ (resp.~$\operatorname{Mod}_{h}(D_{\mathcal X})$, resp.~$\operatorname{Mod}_{rh}(D_{\mathcal X})$) the the full sub-category of $\operatorname{Mod}(D_{\mathcal X})$ consisting category of coherent (resp.~regular, resp.~regular holonomic) $D_{\mathcal X}$-modules.

\subsubsection{Characteristic stack}
Let $\mathcal M$ be a coherent $D_{\mathcal X}$-module with a  good filtration $F_\bullet$.
Following \cite[\S 2.2]{HTT08}, let $\pi: T_{\mathcal X}^\vee\to \mathcal X$ be the cotangent bundle of $\mathcal X$. Since we have $\operatorname{gr}^F D_{\mathcal X}\cong \pi_*\mathcal O_{ T^\vee_{\mathcal X}}$, 
the graded module $\operatorname{gr}^F \mathcal M$ of $\mathcal M$ obtained by $F_\bullet$ is a coherent module over $\pi_*\mathcal O_{T^\vee_{\mathcal X}}$ (this can be established \'etale locally on $\mathcal X$, and so follows from \cite[Prop.~2.1.1]{HTT08}). We call the support of the coherent $\mathcal O_{T^\vee_{\mathcal X}}$-module (with its reduced induced structure)  
$$
\widetilde {\operatorname{gr}^F \mathcal M}:=\mathcal O_{T^\vee_{\mathcal X/\mathbb C}}\otimes_{\pi^{-1}\pi_*\mathcal O_{T^\vee_{\mathcal X}}} \pi^{-1}\operatorname{gr}^F \mathcal M
$$
the \emph{characteristic DM stack of $\mathcal M$},  and denote it by $$\operatorname{Ch}(\mathcal M)\subseteq T^\vee_{\mathcal X}.$$

Clearly, from the definitions, a coherent $D_{\mathcal X}$-module is holonomic if and only if $\dim \operatorname{Ch}(\mathcal M)\le \dim \mathcal X$, extending \cite[Def.~2.3.6]{HTT08} to DM stacks.

\subsection{Hodge modules on DM stacks}

The starting point of our discussion is that Hodge modules, polarisable  Hodge modules, etc.,  satisfy \'etale descent (see \cite[Thm. 2.3]{achar_2013}; the argument is based on the argument for perverse sheaves found in \cite[2.2.19]{BBD82}).  Consequently, we give the following definition of Hodge modules on stacks:

\begin{dfn}[Hodge modules on DM stacks]\label{D:HMXX} 
Let $\mathcal X$ be a smooth separated DM stack locally of finite type over $\mathbb C$.  A \emph{pure Hodge module of weight $w$ (resp.~polarisable pure Hodge module of weight $w$, resp.~Hodge module, resp.~polarisable Hodge module) $\mathsf M$ on $\mathcal X$} is a sheaf of  
pure Hodge modules of weight $w$ (resp.~polarisable pure Hodge module of weight $w$, resp.~Hodge module, resp.~polarisable Hodge module) on the \'etale site of $\mathcal X$.  Morphisms are given by morphisms of sheaves.   
\end{dfn}

Concretely, a pure Hodge module  $\mathsf M$ of weight $w$ (resp.~polarisable pure Hodge module of weight $w$, resp.~Hodge module, resp.~polarisable Hodge module) on $\mathcal X$ is equivalent to the data of a pure Hodge module of weight $w$ (resp.~polarisable pure Hodge module of weight $w$, resp.~Hodge module, resp.~polarisable Hodge module) $M_U$ on $U$ for each \'etale morphism $U\to \mathcal X$ from a smooth variety $U$, together with descent data, as described in \S \ref{S:descentData}.  
Note that in this way, a Hodge module $\mathsf M$ on $\mathcal X$ determines a regular holonomic $D_{\mathcal X}$-module $(\mathcal M,F_\bullet)$ together with a compatible $\mathbb Q$-perverse sheaf $ K$ on the associated analytic stack  $\mathcal X^{\operatorname{an}}$.  We will discuss the perverse sheaf further in another paper where it will be needed.  For the purposes of this paper, we will focus on the associated regular holonomic $D$-module.

We denote by  $\mathsf {HM}(\mathcal X, w)$ the category of pure Hodge modules of weight $w$ on $\mathcal X$, and by $\mathsf {HM^p}(\mathcal X,w)$ the category of polarizable pure Hodge modules of weight $w$ on $\mathcal X$.

\subsection{Hodge modules on stacks form an abelian category }
We establish in this subsection that the categories $\mathsf {HM}(\mathcal X, w)$ and $\mathsf {HM^p}(\mathcal X, w)$ are  abelian categories. All the conditions for these to form abelian categories  can be established  \'etale locally;  the 
existence of kernels and cokernels are a little more involved, and so we include some details here.  Let $\phi:\mathsf M\to \mathsf N$ be a morphism of Hodge modules in $\mathsf {HM}(\mathcal X,w)$.  We want to establish the existence of a kernel and cokernel for $\phi$.  Consider a commutative diagram of \'etale morphisms:
\begin{equation}\label{E:abcatdiag}
\xymatrix{
U'\ar[rr]^{q} \ar[rd]_{p'}& & U \ar[ld]^p\\
&\mathcal X
}
\end{equation}
Let $\phi_U: \mathsf M_U\to \mathsf N_U$ and $\phi_{U'}: \mathsf M_{U'}\to \mathsf N_{U'}$ be the morphisms obtained by restriction, from the definition of the category of Hodge modules as \'etale sheaves.   From the theory of Hodge modules on varieties, each of these morphisms have kernels and cokernels.

We claim that there is a commutative diagram:
$$
\begin{tikzcd}
0 \arrow[r] & q^* \ker (\phi_U) \arrow[r] \arrow[d, "\exists!", dotted] &q^* \mathsf M_U \arrow[r, "q^*\phi_U"] \arrow[d, "\simeq"] & q^* \mathsf N_U \arrow[d, "\simeq"] \arrow[r] & q^* \operatorname{coker}(\phi_U) \arrow[d, "\exists!", dotted] \arrow[r] & 0 \\
0 \arrow[r] & \ker(\phi_{U'}) \arrow[r]                                         & \mathsf M_{U'} \arrow[r,"\phi_{U'}"]                               & \mathsf N_{U'} \arrow[r]                               &  \operatorname{coker}(\phi_{U'}) \arrow[r]                                        & 0
\end{tikzcd}
$$
The square in the middle is a canonical identification, via restriction, and the definition of Hodge modules as \'etale sheaves.  Since $\ker (\phi_{U'})\to \mathsf M_{U'}$ is a kernel for $\phi_{U'}$, it follows that there exists a unique morphism $q^*\ker (\phi_{U})\to \ker (\phi_{U'})$ making the diagram above commute; again from the theory on varieties, one has that this is an isomorphism.  From the theory of Hodge modules on varieties, we also know that $q^*\mathsf N_U\to q^*\operatorname{coker}(\phi_U)$ is a cokernel for $q^*\phi_U$, and so there is a unique morphism $q^*\operatorname{coker}(\phi_U)\to \operatorname{coker}(\phi_{U'})$ making the diagram above commute, and this is an isomorphism.

The existence and uniqueness of the dashed arrows in the diagram above imply that for any \'etale presentation $p:U\to \mathcal X$, the kernel and cokernel of $\phi_U:\mathsf M_U\to \mathsf N_U$ have descent data, and therefore define kernels and cokernels in the category $\mathsf {HM}(\mathcal X,w)$. 
(The argument is the same for polarizable Hodge modules.)

\begin{rem}
We also note that the same argument shows that $\operatorname{Mod}_{rh}(D_{\mathcal X})$ is an abelian category, and that kernels and cokernels, etc., are preserved by the forgetful functor $\mathsf {HM}(\mathcal X,w)\to \operatorname{Mod}_{rh}(D_{\mathcal X})$.
\end{rem}

\begin{dfn}[Tate twists]
 We note here that Hodge modules on stacks have Tate twists, defined as follows.  If $\mathsf M=(\mathcal M,F_{\bullet}, K,w)$ is a (polarizable) Hodge module of weight $w$, then the \emph{$r$-th Tate twist} is defined to the (polarizable) Hodge module   $\mathsf M(r):=(\mathcal M,F_{\bullet-r}, K\otimes_{\mathbb{Q}}(2\pi i)^r\mathbb{Q}, w-2r)$ of weight $w-2r$. 
\end{dfn}

\subsection{Strict support and  torsion sub-modules}
We define the support of a  $D_{\mathcal X}$-module $\mathcal M$ to be the support of the sheaf $\mathcal M$.    Let $\mathcal Z\subseteq \mathcal X$ be an integral closed substack of $\mathcal X$.

\begin{dfn}[Strict support]\label{D:strict-sup}
We say that a $D_{\mathcal X}$-module $\mathcal M$ has \emph{strict support $\mathcal Z$} if the support of every nonzero sub-$D_{\mathcal X}$-module or quotient $D_{\mathcal X}$-module of $\mathcal M$ is equal to $\mathcal Z$.
For a Hodge module  $\mathsf M$ on $\mathcal X$ the  \emph{support of $\mathsf M$} is the support of  its underlying  $D_{\mathcal X}$-module.  We say that a Hodge module $\mathsf M$ has \emph{strict support $\mathcal Z$} if the support of every nonzero sub-Hodge module or quotient  Hodge module of $\mathsf M$ is equal to $\mathcal Z$.
\end{dfn}

\begin{lem}[Sub-modules with support]\label{L:subHdgZZ}
Let $\mathsf  M$ be a polarizable Hodge module on $\mathcal X$ of weight $w$ with associated $D_{\mathcal X}$-module $\mathcal M$.  
For any closed reduced substack $\mathcal Z\subseteq \mathcal X$ there is a corresponding polarizable sub-Hodge module $\mathsf M_{\mathcal Z}\subseteq \mathsf M$ of weight $w$  with associated sub-$D_{\mathcal X}$-module $\mathcal H^0_{\mathcal Z}(\mathcal M)\subseteq \mathcal M$, the sub-sheaf of sections of $\mathcal M$ with support in $\mathcal Z$.    Letting $\mathcal M_{\operatorname{tors}} \subseteq \mathcal M$ be the torsion sub-$\mathcal O_{\mathcal X}$-module, there is   polarizable sub-Hodge module $\mathsf M_{\operatorname{tors}}\subseteq \mathsf M$ with associated filtered $D_{\mathcal X}$-module $\mathcal M_{\operatorname{tors}}$, and
$$
\mathsf M_{\operatorname{tors}}=\bigcup_{\mathcal Z\subsetneq \mathcal X}\mathsf M_{\mathcal Z}.
$$
As $\mathsf {HM}^p(\mathcal X,w)$ is abelian, there is a  quotient polarizable Hodge module  $\mathsf M/\mathsf M_{\operatorname{tors}}$ of weight $w$, which has strict support $\mathcal X$.

\end{lem}

\begin{proof}
Let $\mathsf  M$ be a polarizable Hodge module on $\mathcal X$ of weight $w$ with associated $D_{\mathcal X}$-module $\mathcal M$.  Let $\mathcal Z$ be a reduced closed substack of $\mathcal X$, 
and let $\mathcal H^0_{\mathcal Z}(\mathcal M)\subseteq \mathcal M$ be the sub-sheaf of sections of $\mathcal M$ with support in $\mathcal Z$. 
Now consider a commutative diagram as in \eqref{E:abcatdiag}.

Let $\mathsf M_U$ be the restriction of the \'etale sheaf $\mathsf M$ to $U$, and let $\mathsf {H}^0_{\mathcal Z}(\mathsf M_U)\subseteq \mathsf M_U$ be the sub-Hodge module supported on the pre-image of $\mathcal Z$. We use the same notation on $U'$.   We have a commutative diagram
$$
\xymatrix{
0\ar[r] & q^*\mathsf {H}^0_{\mathcal Z}(\mathsf M_U)\ar[r] \ar@{.>}[d]_{\exists!}&q^*\mathsf M_U\ar[d]_{\simeq} \\
0\ar[r]& \mathsf {H}^0_{\mathcal Z}(\mathsf M_{U'})\ar[r] &  \mathsf M_{U'}
}
$$
The vertical arrow on the right is the canonical identification coming from the definition of $\mathsf M$ as an \'etale sheaf. Denote by $j:W\hookrightarrow U'$ the complement of preimage of $\Z$. The existence of the arrow on the left comes from the fact that the induced map from $q^*\mathsf {H}^0_{\mathcal Z}(\mathsf M_U)$ to the cokernel of the second row, which is a sub-Hodge module of $j_*j^* \mathsf M_{U'}$, is zero. On the level of $\mathcal D$-modules, this follows from the description of $H^0_{\Z}(M_U)$ as the subsheaf of sections killed by a sufficiently high power of the ideal of the preimage of $\Z$ and that $j_*j^* \mathsf M_{U'}$ has no $\Z$-torsion.  A morphism of Hodge modules is zero if and only if it is zero as a morphism of $D$-modules by the faithfulness of the $\operatorname{rat}$ functor (see, e.g., \cite[p.3]{schnelloverview}).  The uniqueness of the arrow comes from the commutativity of the diagram and the  fact that $\mathsf {H}^0_{\mathcal Z}(\mathsf M_{U'})\hookrightarrow  \mathsf M_{U'}$ is a mono-morphism.

The existence and uniqueness of the dashed arrow in the diagram above implies that for any \'etale presentation $p:U\to \mathcal X$, the sub-Hodge module $\mathsf {H}^0_{\mathcal Z}(\mathsf M_{U})\hookrightarrow   \mathsf M_{U}$ on $U$ has descent data, and therefore defines a sub-Hodge module $\mathsf {H}^0_{\mathcal Z}(\mathsf M):=:\mathsf M_{\mathcal Z}\hookrightarrow   \mathsf M$. The rest of the lemma follows easily from this.
\end{proof}

\subsection{Minimal extensions}

Let $j:U\hookrightarrow X$ be an open dense subset of a smooth variety $X$, and let $\mathsf M$ be a pure polarizable Hodge module of weight $w$ on $U$.  In this situation, there is the so-called minimal extension of $\mathsf M$ to $X$, frequently denoted by $j_{!*}\mathsf M$.  For holonomic $D_X$-modules, this is explained in \cite[Def.~3.4.1]{HTT08} in the case where $j$ is assumed to be affine, and in general in \cite[Def.~6.77]{mustata_intro} where it is called the intermediate extension; for perverse sheaves, this is explained in \cite[Def.~8.2.2]{HTT08} with the $\mathbb Q$-structure discussed in \cite[Rem.~8.1.18]{HTT08}. For Hodge modules see \cite[\S 2]{Saito-IntroMHM}, where minimal extension is called the intermediate direct image. 
We also note for convenience in what follows that the functor $\operatorname{rat}: \operatorname{MHM}(X)\to \operatorname{Perv}(\mathbb Q_X)$ is exact  and faithful (see e.g., \cite[p.223]{HTT08} and \cite[p.3]{schnelloverview}), from which it follows that the forgetful functor from Hodge modules to regular holonomic $D$-modules is exact and faithful.

Note that if $\mathsf M$ has strict support $U$, then the minimal extension $j_{!*}\mathsf M$ is the unique Hodge module with strict support $X$ whose restriction to $U$ is $\mathsf M$  (see \cite[Thm.~6.79 vi)]{mustata_intro} for the underlying $D_X$-module and \cite[Prop.~8.2.7 and Prop.~8.2.5]{HTT08} for the underlying perverse sheaf, with the $\mathbb Q$-structure discussed in \cite[Rem.~8.1.18]{HTT08}).

Finally, we note that the minimal extension defines a functor $j_{!*}:\mathsf {HM}(U,w)\to \mathsf {HM}(X,w)$; in particular, morphisms of Hodge modules on $U$ extend uniquely to morphisms of the minimal extensions.  For $D_X$-modules this is \cite[Thm.~6.79 vii)]{mustata_intro} and for perverse sheaves this is explained in \cite[p.204]{HTT08},  with the $\mathbb Q$-structure discussed in \cite[Rem.~8.1.18]{HTT08}.  The minimal extension functor preserves injections and surjections \cite[Thm.~6.79 iii)]{mustata_intro}.

\begin{lem}[Minimal extension]\label{L:stacky-minimal-extension}
Let $\mathcal U\subseteq \mathcal X$ be a dense open substack with complement $\mathcal Z:=\mathcal X-\mathcal U$. 
\begin{enumerate}[label=(\arabic*)]

\item\label{L:StMinExt1} Let $\mathsf N$ be a pure polarizable Hodge module of weight $w$ on $\mathcal U$. 
Then there exists a unique (up to unique isomorphism) pure polarizable Hodge module $\mathsf M$ of weight $w$ on $\mathcal X$  such that 
$\mathsf M|_{\m{U}}\cong \mathsf N$ and $\mathsf M$ has no sub-module or quotient module supported on $\mathcal Z$.
We call the Hodge module $\mathsf M$ the \emph{minimal extension} of $\mathsf N$.

\item \label{L:StMinExtc}If $\mathsf N$ has strict support $\mathcal U$, then  the minimal extension  $\mathsf M$ has strict support $\mathcal X$.

\item \label{L:StMinExt2} Let $\phi: \mathsf N_1\to \mathsf N_2$ be a morphism of pure polarizable Hodge modules on $\mathcal U$.  There exists a unique extension $\tilde \phi:\mathsf M_1\to \mathsf M_2$ of the morphism $\phi$ to the minimal extensions $\mathsf M_1$ and $\mathsf M_2$ of $\mathsf N_1$ and $\mathsf N_2$.  Moreover, the extension $\tilde \phi$ preserves injections and surjections. 
\end{enumerate}
\end{lem}

\begin{proof}\ref{L:StMinExt1}
 Let $p:U\to \mathcal X$ be an \'etale presentation of $\mathcal X$, and let $\tilde j:\widetilde U:=p^{-1}(\mathcal U)\hookrightarrow U$  be the pre-image of $\mathcal U$.  
Let $\mathsf M_U:=\tilde j_{!*}\mathsf N_{\widetilde U}$ be the minimal extension of $\mathsf N_{\widetilde U}$ to $U$.  We claim that $\mathsf M_U$ has descent data induced by the descent data for $\mathsf N_{\widetilde U}$.  Indeed, in the notation from \S \ref{S:DM-HM-stacks}, the descent data for $\mathsf N_{\widetilde U}$ 
is given by  isomorphisms $$\theta_{\alpha\beta}: pr_\beta^*\mathsf N_{\widetilde U}\stackrel{\sim}{\to}pr_\alpha^*\mathsf N_{\widetilde U}$$ on $U_\alpha \times_{\mathcal X}U_\beta$ such that 
$$
(pr_{ij}^*\theta_{ij})\circ (pr_{jk}^*)\theta_{jk} = pr_{ik}^*\theta_{ik}.
$$
From the references mentioned above for varieties, the $\theta_{\alpha\beta}$ extend uniquely to isomorphisms for $\mathsf M_U$.  If the induced co-cycle condition failed for $\mathsf M_U$, then $(pr_{ij}^*\theta_{ij})\circ (pr_{jk}^*)\theta_{jk} - pr_{ik}^*\theta_{ik}\ne 0$ would  imply there were a non-trivial sub-module (or quotient module) of $\mathsf M_U$ supported on $p^{-1}\mathcal Z$, which is not possible.  Thus $\mathsf M_U$ descends to a Hodge module on $\mathsf M$, with the desired property.
The uniqueness of $\mathsf M_U$ then provides the uniqueness of $\mathsf M$.
This also establishes \ref{L:StMinExtc}.

The argument for \ref{L:StMinExt2} is similar to that of \ref{L:StMinExt1}, by reduction to the case of varieties, and is left to the reader.
\end{proof}

\begin{rem}\label{R:MinExtRegHol}
The proof of \Cref{L:stacky-minimal-extension} is made for Hodge modules on $\mathcal X$, but the proofs work identically for regular holonomic $D_{\mathcal X}$-modules using the same argument.
\end{rem}

We will also be interested in the \emph{filtered} regular holonomic $D$-module underling the minimal extension of a Hodge module.  For this we introduce the following terminology.  We say a filtered regular holonomic $D$-module on $\mathcal X$ \emph{admits a $\mathbb Q$-structure \'etale locally} if there exists an \'etale presentation $p:U\to \mathcal X$ such that the filtered regular holonomic $D$-module $(p^*\mathcal M,p^*F_\bullet)$ on the variety $U$ admits a $\mathbb Q$-structure.  We will use the notation $(p^*M,p^*F,K)$ for the filtered regular holonomic $D$-module with $\mathbb Q$-structure, where $K$ is the $\mathbb Q$-perverse sheaf on $U$ giving the $\mathbb Q$-structure.  
   For a filtered regular holonomic $D$-module $(\mathcal M,F_\bullet)$ on $\mathcal X$ that admits a $\mathbb Q$-structure \'etale locally, and an effective divisor $\mathcal D\subseteq \mathcal X$, we say that $(\mathcal M,F_\bullet)$ is \emph{quasi-unipotent (resp.~regular)  along $\mathcal D$} if $(p^*\mathcal M,p^*F_\bullet,K)$ is quasi-unipotent (resp.~regular)  along the divisor  $p^*\mathcal D\subseteq U$ (e.g., \cite[Def.~11.4]{schnelloverview}). One can check that these conditions are independent of the choice of presentation $p:U\to \mathcal X$, as all of the conditions are \'etale local for varieties.

 Note that by definition, the underlying filtered regular holonomic $D$-module of a pure polarizable Hodge module on $\mathcal X$ admits a $\mathbb Q$-structure \'etale locally, and is quasi-unipotent and regular along any effective divisor $\mathcal D\subseteq \mathcal X$.

\begin{lem}\label{L:ExtFilt}
Let $\mathcal U\subseteq \mathcal X$ be a dense open substack,  let $\mathsf N$ be a pure polarizable Hodge module of weight $w$ on $\mathcal U$ with strict support $\mathcal U$, let $(\mathcal N,F_\bullet)$ be the underlying filtered regular holonomic $D$-module on $\mathcal U$,  let $\mathsf M$ be the minimal extension of $\mathsf N$ to $\mathcal X$, and let $(\mathcal M,F_\bullet)$ be the underlying filtered regular holonomic $D$-module on $\mathcal X$.  Let $\mathcal D\subseteq \mathcal X$ be any effective divisor containing the complement $\mathcal X-\mathcal U$.   Then $(\mathcal M,F_\bullet)$ is the unique filtered regular holonomic $D$-module on $\mathcal X$ with strict support $\mathcal X$, which admits a $\mathbb Q$-structure \'etale locally, is quasi-unipotent and regular along $\mathcal D$, and   whose restriction to $\mathcal U$ is isomorphic to $(\mathcal N,F_\bullet)$.  
\end{lem}

\begin{proof}
By construction we have that $(\mathcal M,F_\bullet)$ is a 
filtered regular holonomic $D$-module on $\mathcal X$ with strict support $\mathcal X$, which admits a $\mathbb Q$-structure \'etale locally, is quasi-unipotent and regular along $\mathcal D$, and   whose restriction to $\mathcal U$ is isomorphic to $(\mathcal N,F_\bullet)$.  

 So it suffices to show uniqueness.  Let $(\mathcal M',F_\bullet)$ be another such $D$-module.  From \Cref{R:MinExtRegHol}, we have that $\mathcal M$ and $\mathcal M'$ agree, and so we have reduced to showing that the filtrations agree.  We can check this after an \'etale cover of $\mathcal X$, and so we have reduced to the case of varieties.
  This is discussed in \cite[\S 11]{schnelloverview}, and we review this here.

The question is local, so we may assume we are working on a smooth variety $X$, with an open subset $U\subseteq X$.  The complement $Z= X-U$ is closed, and so there is some non-constant algebraic function $f:X\to \mathbb C$ with zero locus $V(f)$ containing $Z$; for simplicity of notation, we will assume that $Z=V(f)$.  As  $\mathsf M$ is a Hodge module on $X$ with strict support $X$, one has by definition that $(\mathcal M,F_\bullet)$ is a filtered regular holonomic $D$-module on $X$ with strict support $X$,
which admits a $\mathbb Q$-structure, and  is quasi-unipotent and regular along $Z$.  
  The arguments in \cite[\S 11]{schnelloverview} show that such an  $(\mathcal M,F_\bullet)$ is uniquely determined by its restriction to $U$, completing the proof.  
\end{proof}

\subsection{Singular support} Let $\mathsf M$ be a pure Hodge module of weight $w$ on a smooth variety $X$, and assume that $\mathsf M$ has strict support $X$.  Recall that the singular support of $\mathsf M$ is the locus on $X$ where the associated $D_X$-module $M$ is not coherent as an $\mathcal O_X$-module. 
This motivates the following definition:

\begin{dfn}[Singular support]\label{D:sing-supp}
    The \emph{singular support} of a Hodge module $\mathsf M$ on $\mathcal X$, with strict support $\mathcal X$,  
    is the locus where the associated $D_{\mathcal X}$-module $\mathcal M$ is not coherent over $\mathcal O_{\mathcal X}$.
\end{dfn}

As in the case of varieties, we have the following:

\begin{pro} \label{P:Sing-support-is-pure-cod-1}
    Let $\mathsf M$ be a pure polarizable Hodge module of weight $w$ on $\mathcal X$, and assume that $\mathsf M$ has strict support $\mathcal X$.  The singular support of $\mathsf M$ is either empty or of pure codimension $1$.
\end{pro}

\begin{proof}
This is \'etale local, and therefore follows from the case of varieties.  
For lack of a suitable reference, we recall the argument for varieties here. 
Let $\mathsf M$ be a pure polarizable Hodge module of weight $w$ on a variety $X$, and assume that $\mathsf M$ has strict support $X$.
 Let $Z\subseteq X$ be the singular support of $\mathsf M$.  By restricting to an open set, we can remove the divisorial components, and thus assume that the singular support has codimension at least $2$.
 Let $U= X-Z$ be the complement, and let $j:U\hookrightarrow X$ be the inclusion.  
 
 From the references in the section on minimal extensions above, we have that $\mathsf M=j_{!*}(\mathsf M|_U)$, i.e., $\mathsf M$ is the minimal extension of its restriction to $U$. Here we are using the notation $M|_U=j^*M$. By definition of the minimal extension as $j_{!*}(M|_U):=\operatorname{Im}\left(\mathcal{H}^0j_{D!} (M|_U)\to \mathcal{H}^0j_{D*} (M|_U)\right)$, where, to avoid confusion in what follows,  we are using the notation $j_{D!}$ and $j_{D*}$  for the functors  $j_!$ and $j_*$ defined on the derived category of $D$-modules (denoted $\int_{f!}$ and $\int_f$ in \cite{HTT08}), we know that $M\subseteq  \mathcal{H}^0j_{D*} (M|_U)$.  
Now we observe that for an open immersion, the functor $j_{D*}$ agrees with the derived functor $Rj_*$ of the standard push-forward functor $j_*$ of sheaves of $D$-modules (e.g., \cite[Exa.~1.5.22 and p.40]{HTT08}).  Note that while $j_*$ is just the push-forward at the level of sheaves, and so agrees with the functor $j_*$ of sheaves of $\mathcal O_U$-modules, the functor $Rj_*$ is the derived functor in the category of $D$-modules.  
The conclusion is that $ \mathcal{H}^0j_{D*} (M|_U)=  \mathcal{H}^0Rj_{*} (M|_U) =j_*(M|_U)$, so that we have $M\subseteq j_*(M|_U)$, and this is an inclusion of $\mathcal O_X$-modules.  

Next we claim that $j_*(M|_U)$ is a coherent $\mathcal O_X$-module; this essentially follows from \cite[\href{https://stacks.math.columbia.edu/tag/0BK3}{Prop.~0BK3}]{stacks-project} since $M|_U$ is a coherent $\mathcal O_U$-module by assumption, and  $\operatorname{codim}_XZ\ge 2$. In more detail, 
from say \Cref{L:subHdgZZ}, we know that $M|_U$ has  a unique associated prime $x$, the generic point of $U$.  Consequently, we have that $\overline{\{x\}}= X$, which is smooth and projective, so that, assuming $Z$ is non-empty, then  for any $z\in Z\cap \overline{\{x\}}=Z$, there is a unique associated prime $\mathfrak p$ of the completion $\widehat{\mathcal O}_{X,z}$, namely, the minimal prime, and so we have that $\dim (\widehat{\mathcal O}_{X,z}/\mathfrak p) =\dim (\widehat{\mathcal O}_{X,z}) \ge 2$.  It then follows from \cite[\href{https://stacks.math.columbia.edu/tag/0BK3}{Prop.~0BK3}]{stacks-project} that $j_*(M|_U)$ is a coherent $\mathcal O_X$-module. 

Finally, as $M$ is a sub-module of  $j_*(M|_U)$ (and $X$ is locally Noetherian), we have that $M$ is coherent, as well.  In other words, $Z$ is the empty set.
\end{proof}

\section{Preliminaries on non-characteristic morphisms of DM stacks}
We will want to establish some positivity results for Hodge modules on  DM stacks by pulling back the Hodge modules via a finite flat morphism to a smooth projective variety.  In general, pull back   of Hodge modules should be done at the level of the derived category; but for flat non-characteristic morphisms $f$, one has that the pull back   on the associated filtered $D$-modules agrees with the usual pull back   of sheaves of $\mathcal O$-modules. In other words, for flat non-characteristic morphisms, the pull back   $f^!$  is the naive pull back  .  For varieties, this follows from \cite[Thm.~2.4.6]{HTT08} and \cite[Lem.~3.5.3]{saito_RIMS_88}.

As a brief outline, in this section we extend the definition of a non-characteristic morphism to morphisms of  smooth DM stacks (\Cref{D:nonchar}), and explain that for flat non-characteristic morphisms the pull back   of Hodge modules agrees with the naive pull back   (\Cref{P:f*nonchar} and \Cref{L:ffNCpbH}).  In the next section we will use this to  establish that given a Hodge module on a smooth DM stack with projective coarse moduli space, there exists a finite flat non-characteristic cover  with respect to the Hodge module.

\subsection{Non-characteristic morphisms}

Given a morphism $f:\mathcal X\to \mathcal Y$ of smooth separated DM stacks of finite type over $\mathbb C$, we have a diagram
$$
\xymatrix{
\mathcal X \ar[d]_f& \mathcal  X   \times _{\mathcal Y} T^\vee\mathcal Y \ar[l]_<>(0.5){p_1} \ar[d]^{p_2} \ar[r]^<>(0.5){df}& T^\vee\mathcal X\\
\mathcal Y & \ar[l]_p T^\vee\mathcal Y
}
$$

Following \cite[3.5.1]{saito_RIMS_88}, we make the following definition:

\begin{dfn}[Non-characteristic]\label{D:nonchar} 
Given a coherent $D$-module $(\mathcal M,F_{\bullet})$ on $\mathcal Y$, with a choice of good filtration $F_{\bullet}$,  we say that $f$ is \emph{non-characteristic} with respect to $(\mathcal M,F_{\bullet})$ if the morphism
\begin{equation}\label{E:D:nonchar} 
df: p_2^{-1}\operatorname{Ch}(\mathcal M) \to T^\vee\mathcal X
\end{equation}
is finite and   $\mathcal {H}^i (f^{-1}\operatorname{Gr}^F \mathcal M \otimes_{f^{-1}\m{O}_{\Y}}^L \m{O}_{\X})=0$ for all $i\not = 0$.
\end{dfn}

\begin{rem}
We will always be considering non-characteristic morphisms in the situation where $f$ is flat, in which case the condition that $\mathcal {H}^i (f^{-1}\operatorname{Gr}^F \mathcal M \otimes_{f^{-1}\m{O}_{\Y}}^L \m{O}_{\X})=0$ for all $i\not = 0$ is always satisfied.   Note that in \cite[Def.~2.4.2]{HTT08}, where the authors address the non-characteristic condition for coherent $D$-modules without a specific choice of good filtration, the authors do not require Saito's  filtration specific condition, i.e., that   $\mathcal {H}^i (f^{-1}\operatorname{Gr}^F \mathcal M \otimes_{f^{-1}\m{O}_{\Y}}^L \m{O}_{\X})=0$ for all $i\not = 0$.

\end{rem}

  Note that since every fiber $\operatorname{Ch}(\mathcal M)_y\subseteq T_y^*\mathcal Y$ is a cone, the finiteness of \eqref{E:D:nonchar} is equivalent to having 
\begin{equation}\label{E:non-ch}
p_2^{-1}\operatorname{Ch}(\mathcal M)\cap \ker (df) = 0,
\end{equation}
where $0$ indicates the zero section of the pulled back vector bundle  $\mathcal X\times_{\mathcal Y}T^\vee\mathcal Y$;   note also that this condition can be checked \'etale locally on $\mathcal Y$.

Given a $\mathbb C$-point $x$ of $\mathcal X$, let $y=f(x)$.  We have a natural identification 
$$(df)_x= (d_yf): (\mathcal X\times_{\mathcal Y}T^\vee\mathcal Y)_x = T^\vee_y\mathcal Y \longrightarrow T^\vee_x\mathcal X
$$
of $df$ at $x$ with the differential of  cotangent spaces, 
as well as a natural identification of the cones  $(p_2^{-1}\operatorname{Ch}(\mathcal M))_x= \operatorname{Ch}(\mathcal M)_y$.  
We will, consequently, say that \emph{$f$ is non-characteristic with respect to $\mathcal M$ at $x$} if for the cone  $\operatorname{Ch}(\mathcal M)_y\subseteq T_y^*\mathcal Y$ we have 
\begin{equation}\label{E:non-chy}
\operatorname{Ch}(\mathcal M)_y\cap \ker (d_yf)=0 \subseteq T^\vee_y\mathcal Y.
\end{equation}

\subsubsection{Pull back   of $D$-modules via schematic flat non-characteristic morphisms}

Let $f:\mathcal X\to \mathcal Y$ be a morphism of smooth separated DM stacks locally of finite type over $\mathbb C$.  We define the transfer bi-module
$$
D_{\mathcal X\to \mathcal Y}:= \mathcal O_{\mathcal X}\otimes _{f^{-1}\mathcal O_{\mathcal Y}}f^{-1}D_{\mathcal Y},
$$
and observe that this is a left $D_{\mathcal X}$-module and right $f^{-1}D_{\mathcal Y}$-module.  Following \cite[p.21]{HTT08}, we make the following definition:

\begin{dfn}[Naive pull back   of $D$-modules]\label{D:fpb-DM}
Given a (left) $D_{\mathcal Y}$-module $\mathcal M$, we define the \emph{pull back} of $\mathcal M$ as the (left)  $D_{\mathcal X}$-module 
$$
f^*\mathcal M:= D_{\mathcal X\to \mathcal Y}\otimes_{f^{-1}D_{\mathcal Y}} f^{-1}\mathcal M= \mathcal O_{\mathcal X}\otimes _{f^{-1}\mathcal O_{\mathcal Y}}f^{-1}\mathcal M.
$$
\end{dfn}

\begin{rem}
If $\mathcal M$ is quasi-coherent, then $f^*\mathcal M$ is quasi-coherent; in fact we have a functor
$$
f^*:\operatorname{Mod}_{\operatorname{qc}}(D_{\mathcal Y})\longrightarrow \operatorname{Mod}_{\operatorname{qc}}(D_{\mathcal X}).
$$
This is clear from the definition.
\end{rem}

\begin{rem}\label{R:Lf*=f*}
 Let $f:X\to Y$ be a morphism of smooth varieties.  
There is a pull back functor $Lf^*:D^b_{qc}(D_{ Y})\to D^b_{qc}(D_{ X})$, $Lf^*M^\bullet := D_{X\to Y}\otimes^L_{f^{-1}D_{Y}}f^{-1} M^\bullet $, defined at the level of derived categories, where we denote by $D^b(D_{\bullet})$ the derived category of bounded complexes of $D$-modules,  and by $D^b_{qc}(\bullet)$ the full sub-category consisting of complexes whose cohomology sheaves are quasi-coherent (see e.g., \cite[p.32]{HTT08}). 
If $f$ is flat, and $M$ is a quasi-coherent $D_{Y}$-module, then we have an identification
$$
Lf^*\mathcal M=H^0(Lf^* M)= f^*M.
$$
Indeed, as in \cite[Prop.~1.5.8]{HTT08}, we have 
$Lf^*M= \mathcal O_X\otimes ^L_{f^{-1}\mathcal O_Y}f^{-1}M= \mathcal O_X\otimes_{f^{-1}\mathcal O_Y}f^{-1}M$, as $f$ is flat.   
\end{rem}

In general, for pulling back filtered $D$-modules, one should use the pull back   at the level of derived categories (\Cref{R:Lf*=f*}).  However, here we will only need to pull back   along flat non-characteristic morphisms, where for filtered $D$-modules one can still use the naive pull back:

\begin{pro}\label{P:f*nonchar}
Let $f:\mathcal X\to \mathcal Y$ be a schematic flat morphism of smooth separated DM stacks locally of finite type over $\mathbb C$, and let $\mathcal M$ be a coherent $D_{\mathcal Y}$-module.  Assume that $f$ is non-characteristic with respect to $\mathcal M$ (for instance, if $f$ is smooth).  
\begin{enumerate}[label=(\alph*)]
\item\label{E:f*nonchar-coh}  $f^*\mathcal M$ (\Cref{D:fpb-DM}) is a coherent $D_{\mathcal X}$-module.

\item \label{E:f*nonchar-filt} If $F_\bullet \mathcal M$ is a good filtration on $\mathcal M$, then $f^*F_\bullet \mathcal M$ is a good filtration on $f^*\mathcal M$.

\item \label{E:f*nonchar-Ch} $\operatorname{Ch}(f^*\mathcal M)=df(p_2^{-1}\operatorname{Ch}(\mathcal M))$.

\end{enumerate}

\end{pro}

\begin{proof}

\ref{E:f*nonchar-coh}   follows from \ref{E:f*nonchar-filt}, which we show now.  Using that $f^*$ commutes with \'etale base change, and that quasi-coherent sheaves satisfy descent, to show that $f^*F_\bullet \mathcal M$ is a good filtration,  
one is reduced to the case of varieties.  This case follows from the proof of \cite[Thm.~2.4.6(ii)]{HTT08} under the added assumption that $f$ is flat.  Alternatively, 
$f$ being flat implies that \cite[(3.5.1.1)]{saito_RIMS_88} holds, so the result also follows from \cite[Lem.~3.5.3]{saito_RIMS_88}.

\ref{E:f*nonchar-Ch}  This can be checked \'etale locally, and so one is again reduced to the case of varieties.
 It follows from  \cite[Thm.~2.4.6(iii)]{HTT08} that for a morphism of varieties,  if $f$ is non-characteristic, then $\operatorname{Ch}(H^0(Lf^*\mathcal M))\subseteq df(p_2^{-1}\operatorname{Ch}(\mathcal M))$, and in fact equality holds (see \cite[Rem.~2.4.8]{HTT08} and the references).   We saw from \Cref{R:Lf*=f*} that, with the added hypothesis that $f$ is flat, then $Lf^*=f^*$, giving the result.  Alternatively,   $f$ being flat implies in the case of varieties that \cite[(3.5.1.1)]{saito_RIMS_88} holds, so the result also follows from \cite[Lem.~3.5.3]{saito_RIMS_88}.  
\end{proof}

\subsubsection{Pull back   of Hodge modules via schematic finite flat non-characteristic morphisms}

We now turn to the pull back  of Hodge modules:

\begin{lem} \label{L:ffNCpbH}  
Let $f:\mathcal X\to \mathcal Y$ be a finite flat schematic morphism of  smooth separated DM stacks locally of finite type over $\mathbb C$, and let $\mathsf M$ be a pure polarizable Hodge module on $\mathcal Y$ with underlying filtered $D$-module $(\mathcal M,F_\bullet)$.  Assume that $f$ is non-characteristic with respect to $\mathcal M$, and that $\mathcal M$ has strict support $\mathcal Y$.
There is a unique pure polarizable Hodge module $\mathsf M'$ on $\mathcal X$ with underlying filtered $D$-module given by the pull back   $(f^*\mathcal M,f^*F_\bullet)$ from \Cref{P:f*nonchar} with the property that if $\mathcal U\subseteq \mathcal X$ is any open substack such that $f|_{\mathcal U}:\mathcal U\to \mathcal Y$ is \'etale, then $\mathsf M'|_{\mathcal U}=\mathsf M|_{\mathcal U}$ ($=f|_{\mathcal U}^*\mathsf M$).
 Moreover,   $\mathsf M'$ has strict support $\mathcal X$. 
\end{lem}
 
\begin{proof}
It follows from \Cref{P:f*nonchar} that $(f^*\mathcal M,f^*F_\bullet)$ is a holonomic $D$-module on $\mathcal X$. As regularity can be checked \'etale locally (see e.g., \cite[Def.~5.2.10 and Lem.~5.1.23]{HTT08}), to show that $(f^*\mathcal M,f^*F_\bullet)$ is regular, we can reduce  to the case of varieties.  This then follows from, e.g., \cite[Thm.~6.1.5]{HTT08}, where we are using \Cref{R:Lf*=f*} to identify $f^\dagger := Lf^*(-)[\dim \mathcal X-\dim \mathcal Y]= f^*$.    We also have that $(f^*\mathcal M,f^*F_\bullet)$ has strict support $\mathcal X$, since, after pulling back via an \'etale cover  of $ \mathcal Y$, one is reduced to the case of varieties; in the case of varieties, $f^*\mathcal M$ is a pure Hodge module with strict support, as $f$ is non-characteristic (see \cite[Thm. 9.3]{Schnell-SaitoVanishing}).

We now turn to constructing the Hodge module $\mathsf M'$. 
 Let $\mathcal U_{\mathcal Y}\subseteq \mathcal Y$ be an open substack over which $f$ is \'etale, and let $\mathcal U=f^{-1}(\mathcal U_{\mathcal Y})$ be the pre-image in $\X$.  From the definition of a Hodge module as an \'etale sheaf, and the fact that $f|_{\mathcal U}$ is \'etale, we have a well-defined pull back  $\mathsf M|_{\mathcal U}$ with underlying filtered $D$-module $(f^*\mathcal M,f^*F_\bullet)|_{\mathcal U}$.  Cleary, $\mathsf M|_{\mathcal U}$ has strict support $\mathcal U$.

Now define $\mathsf M'$ on ${\mathcal X}$ to be the minimal extension  (\Cref{L:stacky-minimal-extension}) of the Hodge module $\mathsf M|_{\mathcal U}$ on $\mathcal U$ to ${\mathcal X}$. Let $(\mathcal M',F'_\bullet)$ be the underlying regular holonomic filtered $D_{\mathcal X}$-module.  
We need to check that 
 $(\mathcal M',F'_\bullet)\cong (f^*\mathcal M,f^*F_\bullet)$.
 By construction, we have $(\mathcal M',F'_\bullet)|_{\mathcal U}\cong (f^*\mathcal M,f^*F_\bullet)|_{\mathcal U}$.  Let $\mathcal D\subseteq \mathcal X$ be a divisor containing the complement $\mathcal X-\mathcal U$.
 As $\mathsf M'$ is a Hodge module, we have that 
 $(\mathcal M',F_\bullet)$ is a filtered regular holonomic $D$-module on $\mathcal X$ with strict support $\mathcal X$, which admits a $\mathbb Q$-structure \'etale locally, and is quasi-unipotent and regular along $\mathcal D$.  
 We also have that $(f^*\mathcal M,f^*F_\bullet)$ is a filtered regular holonomic $D$-module on $\mathcal X$ with strict support $\mathcal X$. 
 We claim that $(f^*\mathcal M,f^*F_\bullet)$  admits a $\mathbb Q$-structure \'etale locally,  and is quasi-unipotent and regular along $\mathcal D$. To see this, consider the fibered product diagram
$$
\xymatrix{
U'\ar[r]^{p'} \ar[d]_{f'}&\mathcal X \ar[d]_f\\
U \ar[r]^p& \mathcal Y
}
$$
where $p$ is an \'etale presentation of $\mathcal Y$.  By the commutativity of pull back of $\mathcal O_{\mathcal Y}$-modules, we have that $(p'^*f^*\mathcal M,p'^*f^*F_\bullet)=(f'^*p^*\mathcal M,f'^*p^*F_\bullet)$, and the latter is the filtered $D$-module associated to the Hodge module $f'^*(\mathsf M|_U)$ on $U'$.
  The result then follows from \Cref{L:ExtFilt}.
\end{proof}

\begin{dfn}[Finite flat non-characteristic pull back  ]\label{D:ffncpbH}
Let $f:\mathcal X\to \mathcal Y$ be a finite flat schematic morphism of  smooth separated DM stacks locally of finite type over $\mathbb C$, and let $\mathsf M$ be a pure polarizable Hodge module on $\mathcal Y$ with underlying filtered $D$-module $(\mathcal M,F_\bullet)$.  Assume that $f$ is non-characteristic with respect to $\mathcal M$.  We define the \emph{pull back  } $f^*\mathsf M$ of $\mathsf M$ to $\mathcal X$ as the Hodge module $\mathsf M'$ defined in \Cref{L:ffNCpbH}
\end{dfn}

\section{Non-characteristic covers for DM stacks}

In this section, we establish that given a Hodge module on a smooth DM stack with projective coarse moduli space, there exists a finite flat non-characteristic cover, with respect to the Hodge module,  from a smooth projective variety (\Cref{T:nccover}).

\subsection{Some stratifications related to the non-characteristic condition}
We start by considering some stratifications related to the non-characteristic condition.  
Given a morphism $f:\mathcal X\to \mathcal Y$ of smooth separated DM stacks of finite type over $\mathbb C$, 
and  a coherent $D$-module $\mathcal M$ on $\mathcal Y$, then we define a stratification 
$$
\cdots \subseteq \mathcal Y_{\dim \mathcal Y}(\mathcal M)\subseteq \cdots \subseteq \mathcal Y_1 (\mathcal M)\subseteq \mathcal Y_{0}(\mathcal M)=\mathcal Y
$$ 
by the following condition on the $\mathbb C$-points (i.e., $\operatorname{Spec}\mathbb C$-points):
$$
\mathcal Y_i(\mathcal M)(\mathbb C)=\{ y: \dim \operatorname{Ch}(\mathcal M)_y\ge i\}.
$$
The $\mathcal Y_i(\mathcal M)(\mathbb C)$ are closed subsets as $\operatorname{Ch}(\mathcal M)$ is a cone in $T^\vee\mathcal Y$, and so we have that $\mathbb P\operatorname{Ch}(\mathcal M)\subseteq \mathbb PT^\vee\mathcal Y$ is a closed substack, and therefore proper and schematic  over $\mathcal Y$; for proper morphisms, the dimension of the fibers is upper semi-continuous \cite[\href{https://stacks.math.columbia.edu/tag/0D4I}{Lem.~0D4I}]{stacks-project}.  We use the associated induced stack structure to define $\mathcal Y_i(\mathcal M)$ \cite[\href{https://stacks.math.columbia.edu/tag/0509}{Lem.~0509}]{stacks-project}.

\begin{lem}\label{L:Y(M)strat}
If $\mathcal M$ is holonomic, then for all $i$ we have $\dim \mathcal Y_i(\mathcal M) \le \dim \mathcal Y - i$.

\end{lem}

\begin{proof}
The lemma is a consequence of the fact that if $\mathcal M$ is holonomic, then every irreducible component of  $\operatorname{Ch}(\mathcal M)$ has dimension equal to the dimension of $\mathcal Y$. 
Indeed, consider $\mathcal Y_i(\mathcal M)$.   If $\mathcal Y_i(\mathcal M)$ is empty, there is nothing to prove.  If $\mathcal Y_i(\mathcal M)$ is non-empty, suppose that there is an irreducible component with dimension greater than $\dim \mathcal Y-i$.  Then, from the definition of $\mathcal Y_i(\mathcal M)$, there would be a component of $\operatorname{Ch}(\mathcal M)$ of dimension greater than $\dim \mathcal Y$, a contradiction.  Therefore, $\dim \mathcal Y_i(\mathcal M) \le \dim \mathcal Y- i$.  
\end{proof}

We will actually want to modify this stratification to interact with the non-characteristic condition. We define a stratification 
$$
\cdots \subseteq \mathcal Y_{\dim \mathcal Y}(\mathcal M,f)\subseteq \cdots \subseteq \mathcal Y_1 (\mathcal M,f)\subseteq \mathcal Y_{0}(\mathcal M,f)=\mathcal Y
$$ 
by the condition on the $\mathbb C$-points:
$$
\mathcal Y_i(\mathcal M,f)(\mathbb C)=\{ y: \dim \operatorname{Ch}(\mathcal M)_y+\dim \ker (d_yf)\ge i\}.
$$
As before, these define   closed substacks $\mathcal Y_i(\mathcal M,f)$ since $\dim \ker (d_yf)$ is upper semi-continuous, as is $\dim \operatorname{Ch}(\mathcal M)_y$.  
Note that we have the obvious inclusion $$\mathcal Y_i(\mathcal M)\subseteq \mathcal Y_i(\mathcal M,f).$$

\begin{rem} 
To help motivate the definition of this stratification, 
heuristically, the complement $\mathcal Y_i(\mathcal M,f)-\mathcal Y_{i+1}(\mathcal M,f)$, i.e., the locus where $\dim \operatorname{Ch}(\mathcal M)_y+\dim \ker (d_yf)=i$,    is the locus of points $y$ where $\ker (d_yf)$ could be  $(\dim \mathcal Y-i)$ dimensions larger before $f$ would automatically fail the non-characteristic condition at $x$ (with $y=f(x)$) for dimension reasons.  For instance, $\mathcal Y_{\dim \mathcal Y}(\mathcal M,f)$ is the locus of points where, if $\ker df$ were any larger, then $f$ would fail the non-characteristic condition for dimension reasons.  Along these lines, note that if  $\mathcal Y_{\dim \mathcal Y+1}(\mathcal M,f)\ne \emptyset$,  then $f$ cannot be non-characteristic.
\end{rem}

\begin{rem}\label{R:YMfsm}
We note here that if $f$ is smooth, then for all $i$ we have $\mathcal Y_i(\mathcal M)=\mathcal Y_i(\mathcal M,f)$, as the codifferential $d_yf$ is injective.
\end{rem}

\begin{lem}\label{L:nonChYMf}
Suppose that $f$ is non-characteristic with respect to $\mathcal M$.  Then for all $y$ in $\mathcal Y$ we have
$$
\dim \mathcal Y\ge \dim \operatorname{Ch}(\mathcal M)_y+\dim \ker (d_yf).
$$
Moreover, for $y$ in $\mathcal Y$, if $i$ is the largest non-negative integer such that $y$ is in $\mathcal Y_i(\mathcal M,f)$, then 
\begin{equation}\label{E:L:nonChYMf}
\operatorname{Ch}(\mathcal M)_y+\dim \ker (d_yf)= i.
\end{equation}
\end{lem}

\begin{proof}
The first statement is clear from \eqref{E:non-chy} by say projectivizing
 the $(\dim \mathcal Y)$-dimensional vector space $T_y\mathcal Y$ 
 and using that $\operatorname{Ch}(\mathcal M)_y$ is a cone.  The second statement is obvious.
\end{proof}

While we do not have as much control over the dimension of the strata $ \mathcal Y_i(\mathcal M,f)$ as we do for the strata $\mathcal Y_i(\mathcal M)$, we will see that the dimension of the  $ \mathcal Y_i(\mathcal M,f)$ behave  well with respect to hyperplane sections.

\subsection{A Bertini theorem for non-characteristic morphisms}
Our goal is to prove the following Bertini theorem for non-characteristic morphisms:

\begin{pro}\label{P:bertini}
Let $f:\mathcal X\to \mathcal Y$ be a schematic projective morphism of separated DM stacks over $\mathbb C$ with 
quasi-projective coarse moduli spaces $X$ and $Y$, respectively, and assume that $\mathcal Y$ is smooth. 
Let $\mathcal M$ be a holonomic $D$-module on $\mathcal Y$.
Assume further:
\begin{enumerate}[label=(\alph*)]
\item \label{E:bertini-fib} The fiber dimension of $f$ is constant.
\item \label{E:bertini-Q} There exists a dense open smooth subscheme $Q\subseteq \mathcal X$ such that the map $f|_Q: Q\to \mathcal Y$ is flat and surjective.
\item \label{E:bertini-comp} 
The complement $\mathcal X-Q$ has dimension less than the fiber dimension of $f$.
\item  \label{E:bertini-strat} For all $i$ we have $\dim \mathcal Y_i(\mathcal M,f|_Q)\le \dim \mathcal Y-i$.

\item \label{E:bertini-nc} The map $f|_Q$ is non-characteristic with respect to $\mathcal M$.

\end{enumerate}

If the fiber dimension of $f$ is positive, then for any ample line bundle $L$ on $X$ and any embedding $X\subseteq \mathbb P^N$ associated to a sufficiently large tensor power of $L$, a general hyperplane section $X_H=X\cap H$ of $X$ will have the property that setting $\mathcal X_H:=X_H\times_X \mathcal X$, then $X_H$ is the coarse moduli space for $\mathcal X_H$,  and the composition  $f_H:\mathcal X_H\to \mathcal X\to \mathcal Y$ of $f$ with the second projection is a schematic projective morphism satisfying conditions \ref{E:bertini-fib}--\ref{E:bertini-nc} with $Q$ replaced by $Q_H:= X_H\times_X Q= Q\cap H$. 
\end{pro}

The fact that, in the situation of \Cref{P:bertini}, a general hyperplane section of $X$ satisfies the conditions \ref{E:bertini-fib}--\ref{E:bertini-comp} is the content of \cite[Lem.~3.1]{KV04} as it is applied in the proof of \cite[Thm.~2.1]{KV04}.

Consequently, we can focus our attention on \ref{E:bertini-strat} and \ref{E:bertini-nc}.   Note that we are most interested in  \ref{E:bertini-nc}, but  \ref{E:bertini-strat} is helpful for making dimension estimates, and holds automatically if $f$ is smooth, which will eventually be the starting point of our inductive argument.

To this end, let us start by considering  how hyperplane sections interact with the definitions above.  Let $f:\mathcal X\to \mathcal Y$ be as in \Cref{P:bertini}, and assume that we have, as in the proposition, chosen an embedding of the coarse moduli space $X\subseteq \mathbb P^N$.  
Fix a $\mathbb C$-point $x$ of $\mathcal X$ lying in $Q$, and set $y=f(x)$. Abusing notation, we will use the same letter $x$ for the corresponding point of the coarse moduli space $X$, and we note that by assumption we also have $Q\subseteq X$.  We consider the morphism 
$$
H^0(\mathbb P^N,\mathcal O_{\mathbb P^N}(1))\to \mathcal O_{X,x} \to \mathcal O_{X,x}/\mathfrak m_x 
$$
obtained by dividing by any linear form that does not vanish at $x$, and we define
$$
H^0(\mathbb P^N,\mathcal O_{\mathbb P^N}(1))_x:=\ker \left(H^0(\mathbb P^N,\mathcal O_{\mathbb P^N}(1))\to \mathcal O_{X,x}/\mathfrak m_x \right)
$$
to be the sections vanishing at $x$, and note that this is independent of the choice of non-vanishing linear form. 
As the map $H^0(\mathbb P^N,\mathcal O_{\mathbb P^n}(1))\to \mathcal O_{X,x}/\mathfrak m_x$ is surjective (the line bundle $\mathcal O_{\mathbb P^N}(1)$ is very ample when restricted to $X$), we have 
$
\dim H^0(\mathbb P^N,\mathcal O_{\mathbb P^N}(1))_x=N$.

For a non-zero linear form $\ell\in H^0(\mathbb P^N,\mathcal O_{\mathbb P^n}(1))$ we define $X_\ell=X\cap \{\ell =0\}$ to be the associated hyperplane section,  and we use the notation $\mathcal X_{\ell}:=X_{\ell}\times_X\mathcal X$ and $f_{X_\ell}:\mathcal X_\ell\to \mathcal Y$ as introduced in the statement of \Cref{P:bertini}.  Let $\mathcal X^\circ$ be the smooth locus of $\mathcal X$, and similarly let $\mathcal X^\circ_\ell$ be the smooth locus of $\mathcal X_{\ell}$.  Then we have diagrams

$$
\xymatrix{
\mathcal X^\circ  \ar[d]_f& \mathcal  X^\circ   \times _{\mathcal Y} T^\vee\mathcal Y \ar[l]_<>(0.5){p_1} \ar[d]^{p_2} \ar[r]^<>(0.5){df}& T^\vee\mathcal X^\circ & &
 \mathcal X^\circ_\ell  \ar[d]_{f_{X_\ell}}& \mathcal  X^\circ_\ell   \times _{\mathcal Y} T^\vee\mathcal Y \ar[l]_<>(0.5){p_1} \ar[d]^{p_2} \ar[r]^<>(0.5){df_{X_\ell}}& T^\vee\mathcal X^\circ_\ell \\
\mathcal Y & \ar[l]_p T^\vee\mathcal Y 
&&&\mathcal Y & \ar[l]_p T^\vee\mathcal Y
}
$$
Considering also $\operatorname{Ch}(\mathcal M)$, and looking at fibers, we have the following diagram including the fibers of the various objects:
\begin{equation}\label{E:bigCharFib}
\xymatrix{
\operatorname{Ch}(\mathcal M)_y \ar@{^(->}[rd]& & H^0(\mathbb P^N,\mathcal O_{\mathbb P^N}(1))_x  \ar@{->>}[d] \ar@{->>}[rd]& \\
\ker (d_yf)  \ar@{^(->}[r] & T^\vee_y\mathcal Y \ar[r]^<>(0.5){d_yf} \ar@{->}[rd]_{d_yf_{X_\ell}} & T^\vee_x\mathcal X=\mathfrak m_x/\mathfrak m_x^2 \ar@{->>}[d] \ar@{->>}[r]&  (\mathfrak m_x/\mathfrak m_x^2)/T^\vee_y\mathcal Y\\
\ker (d_yf_{X_\ell})  \ar@{^(->}[ru] & &  \mathfrak m_{X_\ell, x}/\mathfrak m_{X_\ell,x}^2
}
\end{equation}
Recall that as $\mathcal Y$ is smooth, and $x$ is a point of $Q$, which is assumed to be smooth and contained in both $X$ and $\mathcal X$, we have that $T^\vee_y\mathcal Y$ is a $(\operatorname{dim}\mathcal Y)$-dimension vector space,  and that we have a natural identification of $(\operatorname{dim}\mathcal X)$-dimensional vector spaces $T^\vee_x\mathcal X=\mathfrak m_x/\mathfrak m_x^2$.   The map $H^0(\mathbb P^N,\mathcal O_{\mathbb P^N}(1))_x \to (\mathfrak m_x/\mathfrak m_x^2)$ is surjective since the 
 line bundle $\mathcal O_{\mathbb P^N}(1)$ is very ample when restricted to $X$. The map $\mathfrak m_x/\mathfrak m_x^2\to \mathfrak m_{X_\ell,x}/\mathfrak m_{X_\ell,x}^2$ is surjective,  induced by the surjective map  $\mathfrak m_x\to \mathfrak m_{X_\ell,x}$ given by setting the linear form $\ell$ to zero.

We now make an elementary observation for later:

\begin{lem}\label{L:dim-Rx}
In the notation above, with $x$ a point of $Q$ and $y=f(x)$,  we have 
\begin{equation}\label{E:L:dim-Rx}
\dim \ker (d_yf )\le \dim \ker (d_yf_{X_\ell}) \le \dim \ker (d_yf)+1,
\end{equation}
and
\begin{equation}\label{E:L:Rx}
R_x:=\ker \left(H^0(\mathbb P^N,\mathcal O_{\mathbb P^N}(1))_x \to  (\mathfrak m_x/\mathfrak m_x^2)/ T_y\mathcal Y\right) 
\end{equation}
\begin{align}
\label{E:L:Rx1}=&\{\ell: [\ell]\in d_yf(T_y\mathcal Y)\subseteq \mathfrak m_x/\mathfrak m_x^2\}\\
\label{E:L:Rx2}=&\{\ell: \ell=0, \text { or } x\in \operatorname{Sing}X_\ell, \text { or }  x\in X_\ell^\circ \text{ and } \dim \ker (d_yf_{X_\ell}) = \dim \ker (d_yf)+1\},\\
\label{E:L:Rx3}=&\{\ell: \ell=0, \text { or } x\in \operatorname{Sing}X_\ell, \text { or }   \dim \ker (d_yf_{X_\ell}) = \dim \ker (d_yf)+1\},
\end{align}
where $[\ell]$ denotes the image of $\ell$ under the morphism $H^0(\mathbb P^N,\mathcal O_{\mathbb P^N}(1))_x\to \mathfrak m_x/\mathfrak m_x^2$, and $X_\ell^\circ$ is the smooth locus of $X_\ell$. 
Moreover, 
$$
\dim R_x= N-(\dim \mathcal X-\dim \mathcal Y) -\dim \ker d_yf.
$$
\end{lem}

\begin{proof}
Let us first check that $\dim \ker (d_yf )\le \dim \ker (d_yf_{X_\ell}) \le \dim \ker (d_yf)+1$.  The first inequality is clear from a diagram chase.  Now, if $\ell$ vanishes to order $2$ at $x$ (i.e., $X_\ell$ is singular at $x$), then the surjection $\mathfrak m_x/\mathfrak m_x^2 \to \mathfrak m_{X_\ell,x}/\mathfrak m_{X_\ell,x}^2$ is an isomorphism, so $\dim \ker (d_yf )= \dim \ker (d_yf_{X_\ell}) $.  On the other hand, if $\ell$ vanishes to order $1$ at $x$, i.e., $X_\ell$ is smooth at $x$, then $\dim  \mathfrak m_{X_\ell,x}/\mathfrak m_{X_\ell,x}^2=\dim X_\ell=\dim X-1$, so the kernel of the surjection $\mathfrak m_x/\mathfrak m_x^2\to \mathfrak m_{X_\ell,x}/\mathfrak m_{X_\ell,x}^2$ has dimension $1$ (spanned by the image of $\ell$).  
A diagram chase shows that $\dim \ker (df_{X_\ell})_y\le \dim (\ker df)_y+1$.

Now let us consider the equality of sets in \eqref{E:L:Rx}--\eqref{E:L:Rx3}. The equality of  \eqref{E:L:Rx} with \eqref{E:L:Rx1} is obvious from the definitions (see diagram \eqref{E:bigCharFib}).
Moving on to \eqref{E:L:Rx2}, let us for brevity set $S_x$ to be the set in \eqref{E:L:Rx2}.
Assume first that $\ell \in R_x$ and $\ell\ne 0$.  There are two possibilities, either the image of $\ell$ in $\mathfrak m_x$ lies in  $\mathfrak m_x^2$, in which case $x\in \operatorname{Sing}X_\ell$, or the image of $\ell$ in $\mathfrak m_x$ does not lie in $\mathfrak m_x^2$, but it does lie in the image of $T_y\mathcal Y$ under $d_yf$.    Then, considering the diagram \eqref{E:bigCharFib}, and the fact that $\ell$ gets sent to zero in the map from $\mathfrak m_x\to \mathfrak m_{X_\ell,x}$,  there must be a strict containment $\ker d_yf\subsetneq \ker d_yf_{X_\ell}$, and we are done by the first part of the lemma.  Thus $R_x\subseteq S_x$.  The opposite containment is similar. The equality of $S_x$ with the last set \eqref{E:L:Rx3} also follows in the same way.

The dimension count for $R_x$ is clear, since the map $H^0(\mathbb P^N,\mathcal O_{\mathbb P^N}(1))_x \to  (\mathfrak m_x/\mathfrak m_x^2)/(T_y\mathcal Y)$ is surjective, and $x$ is a point of $Q$, which is a smooth point of $X$ by assumption. 
\end{proof}

We can now prove a Bertini type theorem for the stratification $\mathcal Y_i(\mathcal M,f)$:

\begin{lem}\label{L:Y(Mf)strat}
Suppose that  $\dim \mathcal Y_i(\mathcal M,f) \le \dim \mathcal Y-i$ for all $i$.  This holds for instance if $\mathcal M$ is holonomic and $f$ is smooth (see \Cref{L:Y(M)strat} and \Cref{R:YMfsm}). Then for a general hyperplane section $X_\ell$ of $X$ and the associated composition $f_{X_\ell}:\mathcal X_\ell \to \mathcal X\to \mathcal Y$, we also have  $\dim \mathcal Y_i(\mathcal M,f_{X_\ell}) \le \dim \mathcal Y-i$ for all $i$.
\end{lem}

\begin{proof} Let $\ell$ be a non-zero linear form in the space $H^0(\mathbb P^N,\mathcal O_{\mathbb P^N}(1))_x$ of linear forms vanishing  a point $x$ of $Q\subseteq \mathcal X$ with $y=f(x)$.  We see from the equality of the sets \eqref{E:L:Rx1} and \eqref{E:L:Rx3} in \Cref{L:dim-Rx} that we will have  $\ker d_yf =\ker d_yf_{X_\ell}$ so long as the image of   $\ell$  in  $\mathfrak m_x/\mathfrak m_x^2$ is not contained in $(d_yf)(T_y\mathcal Y)$.  Since the fiber dimension of $f$ is assumed to be positive, $(d_yf)(T_y\mathcal Y)$ is a linear space strictly contained in  $\mathfrak m_x/\mathfrak m_x^2$.  Therefore, for a general hyperplane section $X_\ell$, we have  $\ker d_yf =\ker d_yf_{X_\ell}$.  The same holds for any finite collection of points, and so for a general  hyperplane section $X_\ell$, we have  $\ker d_yf =\ker d_yf_{X_\ell}$ for a general point of every irreducible component of every $\mathcal Y_i(\mathcal M,f)$.

Now consider an irreducible component $\mathcal Z_i$ of  $ \mathcal Y_i(\mathcal M,f_{X_\ell})$.  
Let $y$ be a general point of $\mathcal Z_i$.  We have $\dim \operatorname{Ch}(\mathcal M)_y$ and $\dim \ker d_yf_{X_\ell}$ are both minimal among points of $\mathcal Z_i$ (both functions are upper semi-continuous).  There are two possibilities, either at the general point $y$ of $\mathcal Z_i$ we have $\dim \ker d_yf_{X_\ell}=\dim \ker d_yf$, or we have $\dim \ker d_yf_{X_\ell}=\dim \ker d_yf+1$. 
From the previous paragraph, the latter does not happen at a general point of $\mathcal Z_i$.  So we must have  $\dim \ker d_yf_{X_\ell}=\dim \ker d_yf$, 
in which case we have $\mathcal Z_i\subseteq \mathcal Y_i(\mathcal M,f)$; in particular, $\dim \mathcal Z_i\le \dim \mathcal Y-i$ from \Cref{L:Y(M)strat}.  
\end{proof}

With this we are now ready to prove \Cref{P:bertini}:

\begin{proof}[Proof of \Cref{P:bertini}]
Let $f:\mathcal X\to \mathcal Y$ be as in \Cref{P:bertini}, and assume that we have, as in the proposition, chosen an embedding of the coarse moduli space $X\subseteq \mathbb P^N$.  
The fact that a general hyperplane section of $X$ satisfies the conditions \ref{E:bertini-fib}--\ref{E:bertini-comp} is the content of \cite[Lem.~3.1]{KV04} as it is applied in the proof of \cite[Thm.~2.1]{KV04}.   Consequently, we can focus our attention on \ref{E:bertini-strat} and \ref{E:bertini-nc}.  Moreover, in \Cref{L:Y(Mf)strat} we showed that \ref{E:bertini-strat} holds for the general hyperplane section.  So it remains to show  \ref{E:bertini-nc}.

As before, for a  $\mathbb C$-point $x$ of $Q\subseteq \mathcal X$, we will denote by $y$ the point  $y=f(x)$.
Now, given a nonzero linear form $\ell$ with zero set passing through $x$, we want to consider the condition that the hyperplane section $ X_\ell$ has the property that $X_\ell$ is singular at $x$, or $X_\ell$ is smooth at $x$ but $f_{\mathcal X_\ell}:\mathcal X_\ell \to \mathcal X\to \mathcal Y$ fails to be non-characteristic at $x$.  We will call these hyperplanes the ``bad'' hyperplane sections at $x$,  and will denote the set by $B_x$:
$$
B_x:=\{\ell : \ell = 0, \text{ or } x\in \operatorname{Sing}(X_\ell), \text{ or } x\in X^\circ_\ell \text{ and }  \operatorname{Ch}(\mathcal M)_y \cap \ker (d_yf_{X_\ell})\ne 0\}$$
$$
 \subseteq H^0(\mathbb P^N,\mathcal O_{\mathbb P^N}(1))_x.
$$
It is important to note that for any $x$ in $Q$, we have by assumption that $f$ is non-characteristic at $x$, so that $ \operatorname{Ch}(\mathcal M)_y \cap \ker (d_yf)=0$.  In particular, if $\ell\in B_x$ is nonzero, we have that  either $X_\ell$ is singular at $x$ or, in order for $ \operatorname{Ch}(\mathcal M)_y \cap \ker (d_yf_{X_\ell})\ne 0$, we must have   $\dim \ker (d_yf_{X_\ell})>\dim \ker (d_yf)$, so that by \eqref{E:L:Rx3} of \Cref{L:dim-Rx},  we have $\ell \in R_x$ \eqref{E:L:Rx}. 
In other words,
$$
B_x\subseteq R_x.
$$
Using the fact that $ \operatorname{Ch}(\mathcal M)_y \cap \ker (d_yf)=0$, then from \eqref{E:L:Rx1} of  \Cref{L:dim-Rx} it is clear that 
$$
\ell\in B_x \iff [\ell] \subseteq (d_yf)(\operatorname{Ch}(\mathcal M)_y) \subseteq \mathfrak m_x/\mathfrak m_x^2,
$$
where $[\ell]$ denotes the image of $\ell$ in $ \mathfrak m_x/\mathfrak m_x^2$.  
In other words, we have a fibered product diagram
$$
\xymatrix{
B_x \ar@{^(->}[r]  \ar@{->>}[d]& R_x \ar@{^(->}[r] \ar@{->>}[d]& H^0(\mathbb P^N,\mathcal O_{\mathbb P^N}(1))_x\ar@{->>}[d]\\
(d_yf)(\operatorname{Ch}(\mathcal M)_y)\ar@{^(->}[r]& (d_yf)(T^\vee_y\mathcal Y)\ar@{^(->}[r]& \mathfrak m_x/\mathfrak m_x^2
}
$$
where the right hand square is a diagram of linear spaces and linear maps, and the left vertical arrow is a surjection of cones.

We now estimate the dimension of $B_x$.
To this end, let $i$ be the largest  integer such that $y$ is in $ \mathcal Y_i(\mathcal M,f)$.   Using \eqref{E:L:nonChYMf}, we have the estimate 
$$\dim (d_yf)(\operatorname{Ch}(\mathcal M)_y)\le \dim \operatorname{Ch}(\mathcal M)_y = i-\dim \ker(d_yf).$$
Therefore, since $\dim ( (d_yf)(T_y\mathcal Y) )= \dim T_y\mathcal Y- \dim \ker(d_yf)= \dim \mathcal Y- \dim \ker(d_yf)$, we see that the codimension of $(d_yf)(\operatorname{Ch}(\mathcal M)_y)\subseteq (d_yf)(T_y\mathcal Y)$ is at least $\dim \mathcal Y-i$.
Since  $B_x$ is the pre-image of the cone $(d_yf)(\operatorname{Ch}(\mathcal M)_y)$ under the surjective linear map $R_x\twoheadrightarrow (d_yf)(T_y\mathcal Y)$, we have that $B_x$ has codimension at least $\dim \mathcal Y-i$, in $R_x$.  We can then use the dimension estimate of $R_x$ in  \Cref{L:dim-Rx} to give the dimension estimate
\begin{align}
\dim B_x&\le \dim R_x-(\dim \mathcal Y -i)\\
&\le [N-(\dim \mathcal X-\dim \mathcal Y)-\dim \ker d_yf] -(\dim \mathcal Y-i)\\ \label{E:dimBx}
 &=N-\dim \mathcal X-\dim \ker d_yf +i.
\end{align}

Now consider the algebraic incidence correspondence:
$$
B:=\{(x,H): x\in Q\subseteq \mathcal X, \ H\in \mathbb PB_x\}\subseteq \mathcal X\times \mathbb PH^0(\mathbb P^N,\mathcal O_{\mathbb P^N}(1)).
$$
 We can estimate the dimension of $B$ by considering the projection onto $\mathcal X$, and then composing with the morphism to $\mathcal Y$.
Let  $B_i$ be an irreducible component of  $B$, and perhaps after re-indexing, assume that for a generic point $x$ in  $B_i$, that $i$ is the largest integer such that $y=f(x)$ is in $\mathcal Y_i(\mathcal M,f)$.  From this we also have that the image of $B_i$ in $\mathcal Y$ is contained in $\mathcal Y_i(\mathcal M,f)$.   In other words, we have chosen $i$ to be the largest integer with the image of $B_i$ contained in  $\mathcal Y_i(\mathcal M,f)$.  
Using the assumption that the fiber dimension of $\mathcal X$ over $\mathcal Y$ is constant, we can then make the following dimension estimate using \eqref{E:dimBx}:
\begin{align*}
\dim B_i&\le \dim \mathcal Y_i(\mathcal M,f) + (\dim_{\mathcal Y}\mathcal X) +\dim \mathbb PB_x\\
& \le (\dim \mathcal Y-i) +(\dim \mathcal X-\dim \mathcal Y) + \left(N-\dim \mathcal X-\dim \ker d_yf +i -1\right)\\
& \le N-\dim \ker d_yf -1<N.
\end{align*}
In particular, no component of $B$ dominates $\mathbb{P}H^0(\mathbb P^N,\mathcal O_{\mathbb P^N}(1))$, and so, for a general hyperplane section $X_\ell$, 
the map $f_{X_\ell}|_{Q_\ell}$ is non-characteristic with respect to $\mathcal M$, where $Q_\ell=Q\cap X_\ell$.  In other words, we have established \ref{E:bertini-nc} for a general $X_\ell$. 
\end{proof}

\subsection{Non-characteristic covers for DM stacks}

\begin{teo}[Non-characteristic cover]\label{T:nccover}
Let $\mathcal M$ be a regular holonomic $D$-module on a smooth separated DM stack $\mathcal X$ of finite type over $\mathbb C$ with (quasi-)projective coarse moduli space.  Then there exists a finite flat   morphism $q:V\to \mathcal X$ from a smooth (quasi-)projective variety $V$ such that $q$ is non-characteristic for $\mathcal M$.
\end{teo}

\begin{proof}
To start, note that the stacks we are considering in this paper are \emph{quotient stacks}, i.e., they are of the form $[W/G]$, where $W$ is a quasi-projective scheme and $G$ is a linear algebraic group (see \cite[Thm. 4.4]{kresch09}), and   so,  in particular, $\mathcal X$ satisfies the hypotheses of \cite[Thm.~2.1]{KV04}.

Our strategy to prove \Cref{T:nccover} is to modify the proof of \cite[Thm.~2.1]{KV04} using \Cref{P:bertini}.  
To be precise, under the hypotheses of \Cref{T:nccover}, Kresch--Vistoli begin their proof by showing that there exists a smooth separated DM stack $\mathcal P$ of finite type over $\mathbb C$ with (quasi-)projective coarse moduli space $P$, and a smooth 
 projective morphism $f:\mathcal P\to \mathcal X$ with an open subscheme $Q\subseteq \mathcal P$ satisfying conditions  conditions \ref{E:bertini-fib}--\ref{E:bertini-comp} of \Cref{P:bertini} (see \cite[Proof of Thm.~2.1, p.4]{KV04} and \cite[Lem.~2.12, \S 4.2]{EHKV01}).

 Since $f$ is smooth, it trivially satisfies \ref{E:bertini-nc} of \Cref{P:bertini}, and since in addition $\mathcal M$ is holonomic, it also satisfies \ref{E:bertini-strat} (\Cref{R:YMfsm} and \Cref{L:Y(M)strat}).

By repeated applications of \Cref{P:bertini}, taking general hyperplane sections, one obtains a  finite flat morphism $q:V\to \mathcal X$ from a smooth (quasi-)projective variety $V$ such that $q$ is non-characteristic for $\mathcal M$.
\end{proof}

\section{Positivity for Hodge modules}

The purpose of this section is to generalize a result of Popa--Wu 
\cite[Thm.~A]{PW16} and a result of Popa--Schnell \cite[Thm.~3.5]{PS17} to the case of smooth DM stacks.   The positivity result, \Cref{T:PW-TA},   generalizes \cite[Thm.~A]{PW16}, which in turn builds on results in \cite{Zuo00negativity, brunebarbe18}.
 \Cref{T:PS-3.5/18.4} uses this positivity result to establish the existence of a type of Viehweg--Zuo sheaf,  
 generalizing \cite[Thm.~3.5]{PS17}, which in turn builds on   \cite{VZ02}.

\subsection{A result of Popa--Wu}

The setup is as follows.  If  $\mathsf M$ is a pure  polarizable Hodge module on a smooth projective variety $X$, then, considering the associated regular holonomic filtered left $D$-module $(M,F_\bullet)$, we have a natural Kodaira--Spencer type $\mathcal O_X$-module homomorphism 
$$
\theta_p:\operatorname{gr}_p^FM\longrightarrow \operatorname{gr}_{p+1}^F M\otimes \Omega^1_X.
$$
We denote by $K_p(\mathsf M):=\ker\theta_p$, which is a coherent $\mathcal O_X$-module, as the $\operatorname{gr}_p^FM$ are.  If $\mathsf M$ has strict support $X$, Popa--Wu show that the torsion-free sheaf $K_p(\mathsf M)^\vee$ is weakly positive for any $p$ \cite[Thm.~A]{PW16}.

Now suppose  $\mathsf M$ is a pure polarizable Hodge module on $\mathcal X$ with associated regular holonomic filtered left $D_{\mathcal X}$-module $(\mathcal M,F_\bullet )$.  In this situation we also have a natural Kodaira--Spencer type $\mathcal O_X$-module homomorphism 
$$
\theta_p:\operatorname{gr}_p^F\mathcal M\to \operatorname{gr}_{p+1}^F \mathcal M\otimes \Omega^1_{\mathcal X}.
$$
As above, we denote by $\mathcal K_p(\mathsf M):=\ker\theta_p$, which is a coherent $\mathcal O_{\mathcal X}$-module.

The following theorem generalizes  \cite[Thm.~A]{PW16} to DM stacks:

\begin{teo}\label{T:PW-TA}
 Let $\mathcal X$ be a smooth proper DM stack over $\mathbb C$ with projective coarse moduli space.  If $\mathsf M$ is a pure polarizable  Hodge module on $\mathcal X$ with strict support $\mathcal X$, then the torsion-free sheaf ${\mathcal K}_p(\mathsf M)^\vee$ is  weakly positive for any $p$.
\end{teo}

\begin{proof}
Let $q:V\to \mathcal X$ be the  finite flat non-characteristic morphism with respect to $\mathsf M$ provided by \Cref{T:nccover}.  The pull back   of $\mathsf M$ to $V$, i.e., $q^*\mathsf M$ (\Cref{D:ffncpbH}), is a pure polarizable Hodge module on $V$ with strict support (\Cref{L:ffNCpbH}). Moreover, if $(\mathcal M,F_\bullet)$ is the underlying filtered $D$-module of $\mathsf M$, then $(q^*\mathcal M,q^*F_\bullet)$ (\Cref{D:fpb-DM} and \Cref{P:f*nonchar}) is the underlying filtered $D$-module of $q^*\mathsf M$ (\Cref{L:ffNCpbH}).

With $\theta_p:\operatorname{gr}_p^F\mathcal M\to \operatorname{gr}_{p+1}^F\mathcal M\otimes \Omega^1_{\mathcal X}$ as above, denote by 
$$
\theta_p(q^*\mathsf M):\operatorname{gr}_p^{q^*F}q^*\mathcal M\to \operatorname{gr}_{p+1}^{q^*F}q^*\mathcal M\otimes \Omega^1_{V}
$$ the Kodaira--Spencer type morphism for $(q^*\mathcal M,q^*F_\bullet)$, with kernel ${\mathcal K}_p(q^*\mathsf M):=\ker \theta_p(q^*\mathsf M)$.  
We have a commutative diagram with exact rows:
$$
\xymatrix@R=1em{
0\ar[r]& {\mathcal K}_p(q^*\mathsf M) \ar[r]&\operatorname{gr}_p^{q^*F}q^*\mathcal M\ar[r]^<>(0.5){ \theta_p(q^*\mathsf M)}& \operatorname{gr}_{p+1}^{q^*F} q^*\mathcal M\otimes \Omega^1_{V}\\
0 \ar[r]&q^*{\mathcal K}_p(\mathsf M)\ar[r] \ar@{->}[u] \ar@{=}[d]&\operatorname{gr}_p^{q^*F}q^*\mathcal M\ar[r]^<>(0.5){q^*\theta_p} \ar@{=}[d] \ar@{=}[u]& \operatorname{gr}_{p+1}^{q^*F} q^*\mathcal M\otimes q^*\Omega^1_{\mathcal X} \ar@{=}[d] \ar@{->}[u]\\
0\ar[r] &q^*{\mathcal K}_p(\mathsf M) \ar[r]&q^*\operatorname{gr}_p^F\mathcal M\ar[r]^<>(0.5){q^*\theta_p}& q^*(\operatorname{gr}_{p+1}^F \mathcal M\otimes \Omega^1_{\mathcal X}).\\
}
$$
The equalities in the bottom half of the diagram come from the fact that the pull back   of the filtration on $\mathcal M$ is the filtration of the pull back   of $\mathcal M$ (\Cref{P:f*nonchar}).   The right vertical arrow 
comes from tensoring the natural inclusion $0\to q^*\Omega^1_{\mathcal X}\to \Omega^1_V$ with $\operatorname{gr}_{p+1}^{q^*F} q^*\mathcal M$.
 The commutativity of the upper right square comes from the definition of the $D$-module structure on the pull back   of a filtered $D$-module.
 A diagram chase shows that the left vertical arrow is an inclusion. (In fact, since the right vertical arrow is generically an isomorphism, a diagram chase implies the left vertical arrow is generically an isomorphism, as well.)

By \cite[Thm.~A]{PW16} applied to $q^*\mathsf M$ on $V$ we have that ${\mathcal K}_p(q^*\mathsf M)^\vee$ is weakly positive.  By dualizing the 
injective morphism $q^*{\mathcal K}_p(\mathsf M) \to  {\mathcal K}_p(q^*\mathsf M)$, we have a generically surjective morphism ${\mathcal K}_p(q^*\mathsf M)^\vee  \to (q^*{\mathcal K}_p(\mathsf M))^\vee$.  
 It follows from \Cref{L:surjWP} that $(q^*{\mathcal K}_p(\mathsf M))^\vee$ is weakly positive.  Finally, since ${\mathcal K}_p(\mathsf M)$ is coherent and $q$ is flat, we have  $(q^*{\mathcal K}_p(\mathsf M))^\vee= q^*({\mathcal K}_p(\mathsf M)^\vee)$ (see e.g., \cite[Proof of Prop.~1.8]{Hart80}).  Therefore,  ${\mathcal K}_p(\mathsf M)^\vee$ is   weakly positive, by definition.
\end{proof}

\subsection{A result of Popa--Schnell}

Given a coherent  $D$-module $(\mathcal M, F_\bullet \mathcal M)$ on $\mathcal X$,  we denote
$$
p(\mathcal M) := \min \{ p :  F_p\mathcal M\ne  0 \}.
$$
In other words, $F_{p(\mathcal M)}\mathcal M = \operatorname{gr}^F
_{p(\mathcal M)} \mathcal M$ 
is the lowest nonzero graded piece in the filtration on $\mathcal M$. If $(\mathcal M, F_\bullet \mathcal M) $ underlies
a Hodge module $\mathsf M$, we also use the notation $p(\mathsf M)$ instead of $p(\mathcal M)$.

Following \cite[Def.~3.4]{PS17}, we make the following  definition:

\begin{dfn}[Large graded submodule]
\label{D:large-graded-sub}
Let $\mathsf M$ be a  pure Hodge module on $\mathcal X$ 
with strict support $\mathcal X$. A graded $\operatorname{gr}^F_\bullet D_{\mathcal X}$-submodule 
$$
\mathcal G_\bullet \subseteq \operatorname{gr}^F_\bullet \mathcal M
$$
is called \emph{large} (with respect to a divisor $\mathcal D$) if there exist a big line bundle $\mathcal A$  and an eﬀective
divisor $\mathcal D$ on $\mathcal X$, together with an integer $\ell \ge 0$, such that:
\begin{itemize}

\item There is a sheaf inclusion $\mathcal A(-\ell \mathcal{D}) \hookrightarrow \mathcal G_{p(\mathsf M)}$

\item the support of the torsion of all of the $\mathcal G_k$ is contained in $\mathcal D$.

\end{itemize}
\end{dfn}

We now use this terminology in the statement of the following theorem, 
which generalizes \cite[Thm.~3.5]{PS17}, establishing the existence of a type of Viehweg--Zuo sheaf:

\begin{teo}\label{T:PS-3.5/18.4}
Let $\mathcal X$ be a smooth proper DM stack over $\mathbb C$ with projective coarse moduli space. Let $\mathsf M$ be a pure polarizable Hodge 
module on $\mathcal X$ with strict support $\mathcal X$ and underlying filtered $D$-module $(\mathcal M, F_\bullet \mathcal M)$,
which is generically a variation of Hodge structure of weight $k$. Assume that there
exists a graded 
$\operatorname{gr}^F_\bullet D_{\mathcal X}$- coherent submodule $\mathcal G_\bullet \subseteq \operatorname{gr}^F_\bullet \mathcal M$ that 
is large with respect to a divisor
$\mathcal D$ (\Cref{D:large-graded-sub}). Then at least one of the following holds:
\begin{enumerate}[label=(\roman*)]
\item \label{E:T:PS-3.5/18.4i} $\mathcal D$ is big.

\item \label{E:T:PS-3.5/18.4ii} There exist integers $s$ and $r$ with $1\le s\le k$ and $1\le r$, a  big coherent sheaf $\mathcal H$ on $\mathcal X$, and an inclusion of sheaves
$$
\mathcal H\hookrightarrow (\Omega_{\mathcal X}^1)^{\otimes s}\otimes \mathcal O_{\mathcal X}(r\mathcal D).
$$
\end{enumerate}
\end{teo}

\begin{proof} The proof is essentially identical to that of \cite[Thm.~3.5]{PS17}.  We reproduce the proof here in order to be able to emphasize the places where one has to adapt the proof in the setting of smooth DM stacks.

Note to begin with that $F_{p(\mathsf M )}\mathcal M$ is a torsion-free sheaf.  Indeed, it suffices to check the torsion-freeness after pull back   to an \'etale cover of $\mathcal X$ by smooth varieties (see \Cref{P:f*nonchar}); the result on smooth varieties is frequently referenced to \cite[Prop.~2.6]{saito_91}, although perhaps \cite[Cor.~(2.9)]{saito_kollar_conj_92} explains this in a little more detail.  It follows that $\mathcal G_{p(\mathsf M )}$ is torsion-free as well. By assumption, there is a big
line bundle $\mathcal A$ on $\mathcal X$ and an integer $\ell \ge 0$, together with an injective sheaf morphism
$$
\mathcal A(-\ell D)\to \mathcal G_{p(\mathsf M)}.
$$

Denoting for simplicity $p= p(\mathsf M )$, the graded $\operatorname{Sym}^\bullet {\mathcal T}_{\mathcal X}$-module structure induces a
sequence of homomorphisms of coherent $\mathcal O_{\mathcal X}$-modules
\begin{equation}\label{E:GG_pChain}
\xymatrix{
0 \ar[r]& \mathcal G_p \ar[r]^<>(0.5){\theta_p}& \mathcal G_{p+1}\otimes \Omega^1_{\mathcal X} \ar[r]^<>(0.5){\theta_{p+1}\otimes 1}& \mathcal G_{p+2}\otimes (\Omega^1_{\mathcal X})^{\otimes 2} \ar[r]& \cdots .
}
\end{equation}

Just as with Hodge modules, we will denote
$$
{\mathcal K}_k={\mathcal K}_k(\mathcal G_\bullet):=\ker \left(\theta_k: \mathcal G_k\longrightarrow \mathcal G_k\otimes \Omega^1_{\mathcal X} \right).
$$
There are obvious inclusions
$$
{\mathcal K}_k\hookrightarrow {\mathcal K}_k(\mathsf M)
$$
which by \Cref{T:PW-TA} and \Cref{L:surjWP} imply that the ${\mathcal K}_k^\vee$ are  weakly positive for all $k$.
 To start making use of this property, note to begin with that given the
inclusion of $\mathcal A(-\ell D)$ into $\mathcal G_p$, there are two possibilities:

The first is that the composition 
$$
\mathcal A(-\ell \mathcal D)\hookrightarrow \mathcal G_p \stackrel{\theta_p}{\to} \mathcal G_{p+1}\otimes \Omega^1_{\mathcal X}
$$
is not injective, so that  $\mathcal A(-\ell \mathcal D)$ has image a torsion sheaf, whose support, by the assumption on $\mathcal G_{p+1}$, must be contained in $\mathcal D$.  It follows that there exists a non-trivial closed substack $\mathcal Z \subseteq \mathcal  X$ supported on $\mathcal D$ such that  $\mathcal A(-\ell \mathcal D) \otimes \mathcal  I_{\mathcal Z} \subseteq  \mathcal K_p$.  This implies that there exists
an integer $r\ge 1$  and an inclusion $\mathcal A(-r\mathcal D)\hookrightarrow \mathcal K_p$, which induces 
a non-trivial
homomorphism
$$
\mathcal K^\vee_p\to \mathcal A^{-1}(r\mathcal D).
$$
Using the weak positivity of
$\mathcal K^\vee_p$ again, we get that $\mathcal A^{-1}(r\mathcal D)$ is pseudo-effective (\Cref{R:WPosLB}).  Now, considering a finite flat morphism $q:V\to \mathcal X$ from a smooth projective variety $V$, we have that $q^*\mathcal A^{-1}(r\mathcal D)$ is  pseudo-effective (\Cref{L:big}).  Since $\mathcal A$ is big, we have that $q^*\mathcal A$ is big (\Cref{L:big}).  Therefore, $q^*r\mathcal D$ must be big, so that $ \mathcal D$ is big (\Cref{L:big}); i.e., we get the condition in \ref{E:T:PS-3.5/18.4i}.

The second possibility is that we have an inclusion
$$
\mathcal A(-\ell \mathcal D)\hookrightarrow \mathcal G_{p+1}\otimes \Omega^1_{\mathcal X}.
$$
We can then repeat the same argument via the morphisms $\theta_s\otimes 1$, with $s\ge p+1$.  The next thing to note however is that there is an $s\le k$
where the inclusions will
have to stop, i.e., such that
$$
\mathcal A(-\ell \mathcal D)\subseteq \mathcal G_{p+s}\otimes (\Omega_{\mathcal X}^1)^{\otimes s} \text { and } \mathcal A(-\ell \mathcal D)\nsubseteq \mathcal G_{p+s+1}\otimes (\Omega_{\mathcal X}^1)^{\otimes s+1}.
$$
Indeed, note that an inclusion $\mathcal A(-\ell \mathcal D) \subseteq \operatorname{gr}_{p+t}^F\mathcal M\otimes (\Omega^1_{\mathcal X})^{\otimes t}$ can only hold so long as $\operatorname{gr}_{p+t}^F\mathcal M$ is not a torsion sheaf.  
Recall, however, that there is a smooth variety $U$ and an \'etale morphism $U\to \mathcal X$ so that the pull back   of $(\mathcal M, F_\bullet \mathcal M)$ to $U$ is a variation of Hodge structure $\mathbf V$ of weight $k$.  Thus on $U$, the pull back   of the sheaves $\operatorname{gr}_{p+t}^F\mathcal M$ coincide with Hodge bundles of $\mathbf V$, and therefore are non-zero only for $t\le k$.  In other words, for $t>k$, the  $\operatorname{gr}_{p+t}^F\mathcal M$ are torsion sheaves.

As above, this implies that there exists some $r\ge 1$ and $1\le s\le k$ such that
$$
\mathcal A(-r\mathcal D)\subseteq \mathcal K_{p+s}\otimes (\Omega^1_{\mathcal X})^{\otimes s}.
$$
We conclude that there exists a non-trivial homomorphism
$$
\mathcal K^\vee_{p+s}\otimes \mathcal A\longrightarrow (\Omega^1_{\mathcal X})^{\otimes s}\otimes \mathcal O_X(r\mathcal D).
$$
Using the  weak positivity of $\mathcal K^\vee_{p+s}$, together with  \Cref{L:pbig-prpty}\ref{L:pbig-prpty-wpb} and \ref{L:pbig-prpty-surj}, 
 and then taking its image under the morphism above 
we obtain an inclusion
$$
\mathcal H\hookrightarrow (\Omega^1_{\mathcal X})^{\otimes s}\otimes \mathcal O_{\mathcal X}(r\mathcal D)
$$
with $\mathcal H$ a  big sheaf on $\mathcal X$; i.e.,   we get the condition in \ref{E:T:PS-3.5/18.4ii}.  
\end{proof}

\section{Positivity for Higgs bundles}

The purpose of this section is to generalize a result of Popa--Wu 
contained in the proof of \cite[Thm.~4.8]{PW16} and a result of Popa--Schnell \cite[Thm.~3.7]{PS17} to the case of smooth DM stacks. 
 The positivity result, \Cref{T:PW-4.8},  generalizes a result from  \cite{PW16}, which in turn builds on results in \cite{Zuo00negativity, brunebarbe18}.
 \Cref{T:PS-3.7/19.1} uses this positivity result to establish the existence of  Viehweg--Zuo sheaves,  
 generalizing \cite[Thm.~3.7]{PS17}, which in turn builds on   \cite{VZ02}.

Proceeding with the discussion, 
we will need a version of \Cref{T:PS-3.5/18.4} 
in the case of (graded logarithmic) Higgs bundles. 
The set-up is as follows: $\mathcal X$ is a smooth proper DM stack over $\mathbb C$ with projective coarse moduli space $X$, and $\mathcal D$ is a simple normal crossings divisor on $\mathcal X$.  
We consider a \emph{graded logarithmic Higgs bundle}
\begin{equation}\label{E:thetaphiggs}
\theta_p:\mathcal E_p \longrightarrow  \mathcal E_{p+1} \otimes \Omega^1_{\mathcal X}(\log \mathcal D).
\end{equation}
To be clear, this is the data of a locally free sheaf  $\mathcal E$ on $\mathcal X$, as well as locally free subsheaves of finite rank $\mathcal E_p \subseteq \mathcal E$ for $p\in \mathbb Z$, such that the natural map $\bigoplus_{p\in \mathbb Z} \mathcal E_p\to \mathcal E$ is an isomorphism, together with an $\mathcal O_{\mathcal X}$-linear map
$$
\theta:\mathcal E \longrightarrow \mathcal E \otimes \Omega^1_{\mathcal X}(\log \mathcal D)
$$
so that the composition $\theta\wedge \theta := (\theta \wedge 1)\circ \theta: 
\mathcal E \to \mathcal E \otimes \bigwedge^2\Omega^1_{\mathcal X}(\log \mathcal D)$ is the zero morphism, and for all $p$, we have $\theta(\mathcal E_p)\subseteq \mathcal E_{p+1}\otimes \Omega^1_{\mathcal X}(\log \mathcal D)$. The condition $\theta\wedge \theta =0$ means that there is a complex
$$
\xymatrix{
0 \ar[r]& \mathcal E \ar[r]^<>(0.5){\theta}& \mathcal E \otimes \Omega^1_{\mathcal X}(\log \mathcal D) \ar[r]^<>(0.5){\theta\wedge 1} 
&  \mathcal E \otimes \bigwedge ^2\Omega^1_{\mathcal X}(\log \mathcal D) \ar[r] & \cdots
}
$$
and the condition $\bigoplus_{p\in \mathbb Z} \mathcal E_p\stackrel{\sim}{\to} \mathcal E$ means that we have induced maps as in \eqref{E:thetaphiggs}, which in turn, for each $p$, induce a complex 
$$
\xymatrix{
0 \ar[r]& \mathcal E_p \ar[r]^<>(0.5){\theta_p}& \mathcal E_{p+1} \otimes \Omega^1_{\mathcal X}(\log \mathcal D) \ar[r]^<>(0.5){\theta_{p+1}\wedge 1} 
&  \mathcal E_{p+2} \otimes\bigwedge^2 \Omega^1_{\mathcal X}(\log \mathcal D) \ar[r] & \cdots.
}
$$

We assume further that the Higgs bundle extends a polarizable variation of Hodge structure of weight $\ell$ on $\mathcal X-D$, meaning 
 that there is an \'etale presentation $P:U\to \mathcal X$ so that denoting by $D'\subseteq U$ the pre-image of $\mathcal D$ under $P$, we have that the graded Higgs bundle $(P^*\mathcal E,P^*\theta)$ on $U$ extends a polarizable variation of Hodge structure of weight $\ell$ on $U-D'$.
Using Tate twists, we will assume that
\begin{equation}\label{E:EE0min}
\mathcal E_p\ne  0 \ \ \ \text{ if and only if } \ \ \ 0\le p\le \ell.  
\end{equation}

Parallel to what was done for Hodge modules, we denote by
\begin{equation}\label{E:KkE}
\mathcal K_p(\mathcal E_\bullet):=\ker \left(\theta_p: \mathcal E_p\to  \mathcal E_{p+1} \otimes \Omega^1_{\mathcal X} (\log \mathcal D) \right).
\end{equation}
Their duals also have a positivity property similar to \Cref{T:PW-TA}, which generalizes a result of \cite{PW16} to DM stacks:

\begin{teo}
\label{T:PW-4.8} Let $\mathcal X$ be a smooth proper DM stack over $\mathbb C$ with projective coarse moduli space, let $\mathcal D$ be an snc divisor on $\mathcal X$, and 
let $(\mathcal E_\bullet, \theta_\bullet)$ be a graded logarithmic Higgs bundle on $\mathcal X$ that 
extends a polarizable variation of Hodge structure of weight $\ell$ on $\mathcal X-\mathcal D$. Then the torsion-free sheaf $\mathcal K_p(\mathcal E_\bullet)^\vee$ is  weakly positive for any $p$. 
\end{teo}

\begin{proof}
The proof in the case where $\mathcal X$ is a smooth projective variety is 
established in \cite{PW16} (see Thm.~4.8 and its proof).   One can deduce the theorem here for DM stacks from the Popa--Wu result by pulling back to a finite flat cover $q:V\to \mathcal X$, and following the steps in the proof of \Cref{T:PW-TA}, replacing the sheaf of differentials with the sheaf of log differentials.

For this argument, one needs to check that there exists a finite flat cover $q:V\to \mathcal X$ such that $q^*\mathcal D$ is snc.  The argument is similar to the proof of \Cref{T:nccover}, i.e., we modify the Bertini type proof of \cite[Thm.~2.1]{KV04}, establishing the existence of the finite flat cover $q:V\to \mathcal X$,  with some additional applications of the standard Bertini theorem.  To be precise, under our hypotheses, Kresch--Vistoli begin their proof by showing that there exists a smooth separated DM stack $\mathcal P$ of finite type over $\mathbb C$ with (quasi-)projective coarse moduli space $P$, and a schematic projective morphism $f:\mathcal P\to \mathcal X$ with an open subscheme $Q\subseteq \mathcal P$ satisfying conditions  conditions \ref{E:bertini-fib}--\ref{E:bertini-comp} of \Cref{P:bertini} (see \cite[Proof of Thm.~2.1, p.4]{KV04} and \cite[Lem.~2.12, \S 4.2]{EHKV01}).  
As $f:\mathcal P\to \mathcal X$ is smooth, we have the added condition that $f^*\mathcal D$ is snc.  For the inductive Bertini argument, we prefer to restate this as follows: for any $1\le r\le \dim \mathcal X$ and any irreducible components $\mathcal D_1,\dots \mathcal D_r$ of $f^*\mathcal D$, the stack-theoretic intersection $\mathcal D_1\cap \cdots \cap \mathcal D_r$ is of pure dimension $\dim \mathcal P-r$, and  smooth when restricted to the subvariety $Q\subseteq \mathcal P$.

After embedding $P\subseteq \mathbb P^N$, 
the fact that a general hyperplane section of $P$ satisfies the conditions \ref{E:bertini-fib}--\ref{E:bertini-comp} is the content of \cite[Lem.~3.1]{KV04} as it is applied in the proof of \cite[Thm.~2.1]{KV04}.   Consequently, we can focus our attention on the condition on the components of $f^*\mathcal D$ inside of $Q$; for the argument, we will be considering $Q\subseteq P$.    Let $H_\ell$ be a general hyperplane in $\mathbb P^N$ and let $P_\ell= P\cap H_\ell$ be the corresponding hyperplane section. 
We will denote by $\mathcal P_\ell:=P_\ell\times _P \mathcal P$, with coarse moduli space $P_\ell$, and set $Q_\ell:=Q\cap H_\ell\subseteq P_\ell$.  
Applying the standard Bertini theorem to the irreducible components of $f^*\mathcal D$, we can assume that the restriction of any irreducible component $\mathcal D_i$ of $f^*\mathcal D$ to $P_\ell$ is smooth when restricted to $Q$.  
More generally, for any $1\le r\le \dim \mathcal X$ and any irreducible components $\mathcal D_1,\dots \mathcal D_r$ of $f^*\mathcal D$, applying the standard Bertini theorem to the intersection $\mathcal D_1\cap \cdots \cap \mathcal D_r$ inside of $Q$, one has that the restriction of $\mathcal D_1\cap \cdots \cap \mathcal D_r$ to $Q_\ell$ is of pure dimension $\dim \mathcal P-r-1$, and  smooth.   Inductively, we obtain the finite flat cover $q:V\to \mathcal X$ such that $q^*\mathcal D$ is snc.  
\end{proof}

We also
consider a graded submodule $\mathcal F_\bullet \subseteq \mathcal E_\bullet$ having the property that 
\begin{equation}\label{E:FsubEprop}
\theta_p(\mathcal F_p)\subseteq \mathcal F_{p+1}\otimes \Omega^1_{\mathcal X}(\log \mathcal B)
\end{equation}
for some divisor $\mathcal B\subseteq \mathcal D$.
Note that since the $\mathcal E_p$ are vector bundles, the sheaves $\mathcal F_p$ are
automatically torsion-free. As with logarithmic graded Higgs bundles, we define 
$$
{\mathcal K}_p(\mathcal F_\bullet):=\ker \left(\theta_p: \mathcal F_p\longrightarrow \mathcal F_{p+1}\otimes \Omega^1_{\mathcal X}(\log \mathcal B) \right).
$$

Similar to  \Cref{D:large-graded-sub}, and following \cite[\S 3.3]{PS17}, we make the following definition:

\begin{dfn}[Large graded submodule of a Higgs bundle]\label{D:large-sub-Higgs}
A graded submodule $\mathcal F_\bullet \subseteq \mathcal E_\bullet$ of a graded logarithmic Higgs bundle $\mathcal E_\bullet$ 
satisfying \eqref{E:FsubEprop} for some divisor  $\mathcal B\subseteq \mathcal D$ 
is \emph{large} if there exists a big line bundle $\mathcal A$ such that $\mathcal A\subseteq \mathcal F_0$.
\end{dfn}

In this situation, we will say that $\mathcal E_\bullet$ admits a large graded submodule $\mathcal F_\bullet$ for the divisor $\mathcal B\subseteq \mathcal D$.  
The next theorem  generalizes  \cite[Thm.~3.7]{PS17} to the case of DM stacks, establishing the existence of Viehweg--Zuo sheaves:

\begin{teo}
\label{T:PS-3.7/19.1}
Let $\mathcal X$ be a smooth proper DM stack over $\mathbb C$ with projective coarse moduli space, let $\mathcal D$ be an snc divisor on $\mathcal X$, and 
let $(\mathcal E_\bullet, \theta_\bullet)$ be a graded logarithmic Higgs bundle on $\mathcal X$ that 
extends a polarizable variation of Hodge structure of weight $\ell$ on $\mathcal X-\mathcal D$. 
If $\mathcal E_\bullet$ admits a large graded submodule for some  divisor $\mathcal B\subseteq \mathcal D$ 
(\Cref{D:large-sub-Higgs}), 
then there exist a  big coherent sheaf $\mathcal H$ on $\mathcal X$
and an integer  $s$ with  $1\le s\le \ell$,  together with an inclusion
$$
\xymatrix{\mathcal H \ar@{^(->}[r] & \left( \Omega^1_{\mathcal X}(\log \mathcal B)\right)^{\otimes s}.
}
$$
\end{teo}

\begin{proof}
The proof is essentially the same as that of \Cref{T:PS-3.5/18.4},
replacing the sequence of coherent $\mathcal O_{\mathcal X}$-module homomorphisms in \eqref{E:GG_pChain}  with the sequence of coherent $\mathcal O_{\mathcal X}$-module homomorphisms  
\begin{equation}\label{E:FF_pChain}
\xymatrix@C=1.65em{
0 \ar[r]& \mathcal F_0 \ar[r]^<>(0.5){\theta_0}& \mathcal F_{1} \otimes \Omega^1_{\mathcal X}(\log \mathcal B) \ar[r]^<>(0.5){\theta_{1}\otimes 1} 
&  \mathcal F_{2} \otimes (\Omega^1_{\mathcal X}(\log \mathcal B))^{\otimes 2}\ar[r]^<>(0.5){\theta_2\otimes 1} & \cdots \ar[r]^<>(0.5){\theta_{\ell-1}\otimes 1}& \mathcal F_\ell \otimes (\Omega^1_{\mathcal X}(\log \mathcal B))^{\otimes \ell}\ar[r]^<>(0.5){\theta_\ell \otimes 1}& 0.
}
\end{equation}
Indeed, considering that all of the sheaves above are torsion-free, that $\theta_\ell \otimes 1$ is the zero map, and that there is a big line bundle $\mathcal A$ contained in $\mathcal F_0$, there must be some $s$ with $1\le s\le \ell$ so that $$\mathcal A\subseteq \mathcal K_s(\mathcal F_\bullet)\otimes (\Omega^1_{\mathcal X}(\log \mathcal B))^{\otimes s}.$$
We conclude that there exists a non-trivial homomorphism
$$
\mathcal K_{s}(\mathcal F_\bullet)^\vee \otimes \mathcal A\longrightarrow (\Omega^1_{\mathcal X}(\log \mathcal B))^{\otimes s}
$$
and, using the  weak positivity of $\mathcal K_s(\mathcal F_\bullet)^\vee$ (this sheaf admits a generically surjective morphism from $\mathcal K_s(\mathcal E_\bullet)^\vee$, which is weakly positive by \Cref{T:PW-4.8}),  together with \Cref{L:pbig-prpty}\ref{L:pbig-prpty-wpb} and \ref{L:pbig-prpty-surj}, 
 taking the  image of the morphism above, 
we obtain an inclusion
$$
\xymatrix{\mathcal H \ar@{^(->}[r] & \left( \Omega^1_{\mathcal X}(\log \mathcal B)\right)^{\otimes s}
}
$$
with $\mathcal H$ a  big sheaf on $\mathcal X$. 
\end{proof}

\begin{rem}\label{R:PThmB}
We note that \Cref{TA:PS-3.7/19.1} is a direct consequence of \Cref{T:PS-3.7/19.1} and \Cref{L:pbig-prpty}\ref{L:pbig-prpty-det}. 
\end{rem}


\ifArxiv

\appendix

\section{Strict positivity on stacks}

For various reasons, in our discussion of positivity on stacks, we prefer to avoid working with global generation conditions directly on stacks.  For completeness, however, we introduce in this appendix some related notions of positivity on stacks, which   we will call ``strict'' positivity, that use the global generation condition.

\subsection{Strict weak positivity}

Motivated by the definition for varieties, we say that a quasi-coherent sheaf $\mathcal F$ on $\mathcal X$ is \emph{generated by global sections at each point of a Zariski open substack $\mathcal U\subseteq \mathcal X$} if the canonical morphism of sheaves
\begin{equation}
H^0(\mathcal X,\mathcal F)\otimes_{\mathbb C}\mathcal O_{\mathcal X}\to \mathcal F
\end{equation}
is surjective when restricted to $\mathcal U$.  Following the definition for varieties:

\begin{dfn}[Strictly weakly positive sheaves]
Let $\mathcal F$ be a torsion-free coherent sheaf on $\mathcal X$.  We say that $\mathcal F$ is \emph{strictly weakly positive over a Zariski open substack $\mathcal U\subseteq \mathcal X$}   if for every integer $\alpha > 0$ and every line bundle $\mathcal A$ on $\mathcal X$ that has a positive tensor power that descends to an ample line bundle on $X$, there is an integer $\beta >0$ 
such that
$$
(\operatorname{Sym}^{\alpha \beta} \mathcal F)^{\vee \vee}\otimes \mathcal A^{\beta} 
$$
is generated by global sections at each point of  $\mathcal U$. We say that $\mathcal F$ is \emph{strictly weakly positive} if such an open substack $\mathcal U \ne \emptyset$ exists.
\end{dfn}

\begin{lem}\label{L:WPpullback}
For any surjective morphism $q':V'\to \mathcal X$ from a smooth projective variety $V'$, if $\mathcal F$ is strictly weakly positive over $\mathcal U\subseteq X$ then $q'^*\mathcal F$ is weakly positive over $q'^{-1}(\mathcal U)$.  In particular, if $\mathcal F$ is strictly weakly positive, then it is weakly positive. 
\end{lem}

\begin{proof} 
The proof is the same  as \cite[3.4 a)]{Vhilb1}. 
The first observation is \cite[Rem.~(1.3) ii)]{Vkod1}, that it suffices to check the condition of weak positivity of $q'^*\mathcal F$ for a single ample line bundle on $V'$.  
The next key point is \cite[3.3 d)]{Vhilb1}, that we can work on the open locus where $\mathcal F$ is locally free, which has complement of codimension at least $2$.  
In short, this implies that taking symmetric powers and then double duals commutes with pull back  , etc. 
 
 So let us assume, for a fixed  $\mathcal A$ on $\mathcal X$ that has some positive tensor power that descends to an ample line bundle on $X$, and a given $\alpha>0$, that $(\operatorname{Sym}^{\alpha\beta}\mathcal F)^{\vee\vee}\otimes \mathcal A^\beta$ is generated by global sections over $\mathcal U\subseteq \mathcal X$ for some integer $\beta >0$.  
We then have a morphism
$$
\xymatrix{
\bigoplus q'^*\mathcal O_{\mathcal X}\ar[r]  &  q'^*((\operatorname{Sym}^{\alpha\beta}\mathcal F)^{\vee\vee}\otimes \mathcal A^\beta)
}
$$
The map above is surjective over $q'^{-1}(\mathcal U)$, since pull back   is right exact.   By virtue of the fact that $q'^*\mathcal A$ is ample, and $q'^*\mathcal O_{\mathcal X}=\mathcal O_V$, we have  that $q'^*\mathcal F$ is weakly positive over $q'^{-1}(\mathcal U)$.  
\end{proof}

\begin{lem}[Strictly weakly positive versus weakly positive]\label{R:PWP=>WP}
If $\mathcal X$ has generically trivial stabilizers, then 
 $\mathcal F$ is strictly weakly positive if and only if $\mathcal F$ is  weakly positive.  
 \end{lem}
 
\begin{proof}
One direction follows from \Cref{L:WPpullback}.  
Let us now show that if $\mathcal F$ is weakly positive over $\mathcal U$ then it is strictly weakly positive over some open substack $\mathcal U'\subseteq \mathcal U$.  
The proof is essentially identical to the proof of \cite[3.4 b)]{Vhilb1}.  
As $q:V\to \mathcal X$ is finite and flat, the trace homomorphism $q_*\mathcal O_V\to \mathcal O_{\mathcal X}$ splits the natural inclusion
$\mathcal O_{\mathcal X}\to q_*\mathcal O_V$ (e.g., \cite[p.53]{Vmoduli}).
 Let $A$ be an ample line bundle on $X$, and take $\eta \gg 0$ so that there is a surjection $\bigoplus \mathcal O_X\twoheadrightarrow \pi_*q_*\mathcal O_V \otimes A^\eta$.  We then apply $\pi^*$ and consider the composition
\begin{equation}\label{L:WPchar-pfeq1}
\bigoplus \mathcal O_{\mathcal X}\twoheadrightarrow  \pi^*\pi_*q_*\mathcal O_V\otimes \pi^*A^\eta \to q_*\mathcal O_V \otimes \pi^*A^\eta, 
\end{equation}
where the second morphism is the canonical morphism from adjunction tensored by $\pi^*A^\eta$.  The first morphisms is a surjection, since $\pi^*$ is right exact. The second morphism is generically an isomorphism, since  $\pi^*\pi_*\mathcal O_{\mathcal X}\to \mathcal O_{\mathcal X}$ is an isomorphism, and over the  open substack $\mathcal U'\subseteq \mathcal U$ where $\mathcal X$ has trivial stabilizers,  $q_*\mathcal O_V$ is locally free in the Zariski topology ($q$ is finite and flat and $\pi\circ q$ and $q$ agree on the locus where there are trivial stabilizers).

Now suppose that there is a morphism
$$
\bigoplus \mathcal O_{V}\to (\operatorname{Sym}^{\alpha\beta}q^*\mathcal F)^{\vee\vee}\otimes q^* \pi^*A^\beta
$$
that is a surjection over an open subset $q^{-1}(\mathcal U)$.  Applying $q_*$ and tensoring by $\pi^*A^\eta$, we obtain 
$$
\bigoplus q_*\mathcal O_{V} \otimes \pi^*A^\eta \to (\operatorname{Sym}^{\alpha\beta}\mathcal F)^{\vee\vee}\otimes  \pi^*A^{\beta +\eta} \otimes q_*\mathcal O_V \to (\operatorname{Sym}^{\alpha\beta}\mathcal F)^{\vee\vee}\otimes  \pi^*A^{\beta +\eta},
$$
where for the second homomorphism we are using the trace $q_*\mathcal O_V\to \mathcal O_{\mathcal X}$.  The first homomorphism remains surjective over $\mathcal U$ since $R^1q_*=0$ on coherent sheaves, as $q$ is finite.  The second morphism is surjective since the trace is.  Finally, we conclude using the generic surjection \eqref{L:WPchar-pfeq1}.
\end{proof}

\subsection{Strict bigness}

We can define strictly big sheaves on stacks by directly generalizing the definition of big sheaves on varieties in the following way:

\begin{dfn}[Strictly big sheaves]
A torsion-free coherent sheaf $\mathcal F$ on $\mathcal X$ is \emph{strictly big}
if there exists a line bundle $\mathcal A$ on $\mathcal X$ such that some positive tensor power descends to an ample line bundle on $X$, a natural number $\nu$, and an inclusion
\begin{equation}\label{E:bigF-defSt}
\bigoplus ^{\operatorname{rank}\operatorname{Sym}^\nu F}\mathcal A \hookrightarrow (\operatorname{Sym}^\nu \mathcal F)^{\vee \vee}.
\end{equation}
 
\end{dfn}

\begin{lem}
If $\mathcal F$ is strictly big, then $\mathcal F$ is big.
\end{lem}

\begin{proof}
The poof  is the same as that of  \Cref{L:big-prpty}\ref{L:big-prpty-ff}, using the finite flat   morphism $q:V\to \mathcal X$.  
\end{proof}

\begin{rem}

Similar to \Cref{R:bigShdesc}, we note that if $\mathcal F$ descends to $F$ on $X$, then  if $F$ is big, then    $\mathcal F$ is strictly big; 
the argument is the same as \Cref{L:big-prpty}\ref{L:big-prpty-ff}, using that the finite morphism $\pi:\mathcal X\to X$ is generically flat (in particular, flat over the smooth locus of $X$; see \S \ref{S:DM-intro}).
\end{rem}

\fi 

 \bibliographystyle{amsalpha}
 \bibliography{mhm_bib}

\end{document}